\documentclass[11pt]{article}

\usepackage{latexsym,amsmath,amssymb,amsfonts}
\usepackage{epsfig,graphics,color}
\usepackage{hyperref}
\hypersetup{
    colorlinks=true,
    linkcolor=blue,
    filecolor=blue,      
    urlcolor=cyan,
    citecolor=blue,
}

\usepackage{rotating}
\usepackage{stackrel}
\usepackage{amsmath, amsthm, amscd, amsfonts, graphicx,makeidx,amssymb,mathrsfs, fancyhdr,float,capt-of}
\usepackage[all]{xy}
\usepackage{verbatim}
\usepackage{todonotes}

\usepackage{pgf,tikz}
\usepackage[mathcal]{euscript}
\usepackage{graphicx,epsfig}

\usepackage{pifont}
\usepackage{color}
\usepackage{tikz-cd}
\RequirePackage[mathcal,mathbf]{euler}

\newtheorem{proposicion}{Proposition}[section]
\newtheorem{lema}[proposicion]{Lemma}
\newtheorem{teorema}[proposicion]{Theorem}
\newtheorem{corolario}[proposicion]{Corollary}

\theoremstyle{definition}
\newtheorem{observacion}[proposicion]{Remark}
\newtheorem{definicion}[proposicion]{Definition}
\newtheorem{defn}[proposicion]{Definition}
\newtheorem{ejemplo}[proposicion]{Example}
\newtheorem{example}[proposicion]{Example}
\newtheorem{noter}[proposicion]{Note}

\newcommand{\bobs}{\begin{observacion}}
\newcommand{\eobs}{\end{observacion}}

\newcommand{\beq}{\begin{equation}}
\newcommand{\eeq}{\end{equation}}
\newcommand{\bprop}{\begin{proposicion}}
\newcommand{\eprop}{\end{proposicion}}
\newcommand{\blema}{\begin{lema}}
\newcommand{\elema}{\end{lema}}
\newcommand{\bejem}{\begin{ejemplo}}
\newcommand{\eejem}{\end{ejemplo}}
\newcommand{\bteo}{\begin{teorema}}
\newcommand{\eteo}{\end{teorema}}
\newcommand{\bdefin}{\begin{definicion}}
\newcommand{\edefin}{\end{definicion}}
\newcommand{\benum}{\begin{enumerate}}
\newcommand{\eenum}{\end{enumerate}}
\newcommand{\bcor}{\begin{corolario}}
\newcommand{\ecor}{\end{corolario}}
\newcommand{\bmat}{\begin{matrix}}
\newcommand{\emat}{\end{matrix}}
\newcommand{\barr}{\begin{array}}
\newcommand{\earr}{\end{array}}
\newcommand{\bcas}{\begin{cases}}
\newcommand{\ecas}{\end{cases}}
\newcommand{\bcen}{\begin{center}}
\newcommand{\ecen}{\end{center}}
\newcommand{\bdem}{\begin{proof}}
\newcommand{\edem}{\end{proof}}
\newcommand{\To}{\longrightarrow}
\newcommand{\ma}{\mathcal}
\newcommand{\op}{\operatorname}

\newcommand{\RR}{{\mathcal R}}

\newcommand{\ben}{\begin{equation}}
\newcommand{\een}{\end{equation}}
\newcommand{\bena}{\begin{equation*}}
\newcommand{\eena}{\end{equation*}}

\newcommand{\integer}{\ensuremath{{\mathbb Z}}}
\newcommand{\ZZ}{\ensuremath{{\mathbb Z}}}

\def\RR{\mathbb{R}}
\def\ZZ{\mathbb{Z}}

  \title{$\mathbb{Z}_k$-Stratifolds}
\author{Andr\'es Angel\thanks{ja.angel908@uniandes.edu.co}, Carlos Segovia \thanks{csegovia@matem.unam.mx}
and
 Arley Fernando Torres \thanks{arley.torres@uexternado.edu.co / torres.af@javeriana.edu.co}}

\begin{document}

\maketitle

\begin{abstract}Generalizing the ideas of $\ZZ_k$-manifolds from Sullivan and stratifolds from Kreck, we define $\mathbb{Z}_k$-stratifolds. 
We show that the bordism theory of $\ZZ_k$-stratifolds is sufficient to represent all homology classes of a $CW$-complex with coefficients in $\ZZ_k$.
We present a geometric interpretation of the Bockstein long exact sequences and the Atiyah--Hirzebruch spectral sequence for $\ZZ_k$-bordism ($k$ an odd number). 
Finally, for $p$ an odd prime, we give geometric representatives of all classes in $H_*(B\mathbb{Z}_p ;\mathbb{Z}_p)$ using $\ZZ_p$-stratifolds.
\end{abstract}

\section{Introduction}
Various geometric models of homology classes use the notion of bordism. 
For instance, Baas \cite{baas} constructs a generalized homology theory using the bordism of manifolds with singularities.
Buoncristiano--Rourke--Sanderson \cite{AGA} gives a geometric treatment of generalized homology.
Certain singular spaces called $\ZZ_k$-manifolds were introduced initially by Sullivan \cite{Sullivan3,Sullivan1,Sullivan4}, although Morgan--Sullivan \cite{sullivan} gives the first formal study of this subject. 
The theory of $\ZZ_k$-manifolds gives a geometric model for $\ZZ_k$-homology classes, but Sullivan pointed out that $\ZZ_k$-manifolds are not general enough to represent $\ZZ_k$-homology. For example, the generator of $H_8(K(\ZZ, 3);\ZZ_3)$ is not represented by a $\ZZ_3$-manifold, see \cite{Sullivan2}. Moreover, Brumfiel \cite{Brum} shows that the non-zero classes in $H_{2p}(K(\ZZ_p,1);\ZZ_p)$ cannot be represented by $\ZZ_p$-manifolds whenever $p$ is prime. 
In fact, in this work, we show that for every $i\geq 3$ there exists a class $\alpha_{2i}\in H_{2i}(B\ZZ_p;\ZZ_p)$ that cannot be represented by $\ZZ_p$-manifolds.
Thus a geometric model is needed to represent every homology class with $\ZZ_k$-coefficients. For this purpose, 
we focus on the theory of stratifolds developed by Kreck \cite{kreck}, where the homology with $\ZZ$-coefficients and $\mathbb{Z}_2$-coefficients are represented by the bordism theories of stratifold homology $SH_*(X)$ and stratifold homology with $\ZZ_2$-coefficients (only works for $\ZZ_2$-coefficients). 

We consider the generalized homology theory of bordism of $\ZZ_k$-manifolds with continuous maps to $X$, denoted by $\Omega_*(X;\ZZ_k)$. There is a long exact sequence satisfying the commutative diagram 
\begin{equation}\label{diag1}
    \xymatrix{\cdots \ar[r] & \Omega_n(X) \ar[r]^{\times k}\ar[d]^h &  \Omega_n(X) \ar[r]^r \ar[d]^h & \Omega_n(X;\ZZ_k)\ar[r]^\delta\ar[d]^{h_{\ZZ_k}} & \Omega_{n-1}(X)\ar[r]\ar[d]^h &\cdots\\
    \cdots \ar[r] & H_n(X) \ar[r]^{\times k} &H_n(X)\ar[r]^r & H_n(X;\ZZ_k)\ar[r]  & H_{n-1}(X)\ar[r] &\cdots}
\end{equation}
where $\delta: \Omega_*(X;\ZZ_k)\rightarrow\Omega_{n-1}(X)$ is the Bockstein homomorphism, $r:\Omega_n(X)\rightarrow\Omega_n(X;\ZZ_k)$ is the reduction homomorphism obtained by considering a closed manifold as a $\ZZ_k$-manifold with empty Bockstein, and 
$h_{\ZZ_k}: \Omega_*(X;\mathbb{Z}_k) \rightarrow H_*(X;\mathbb{Z}_k)$ is the Hurewicz homomorphism provided by the existence of fundamental $\ZZ_k$-homology classes. 

Generalizing the ideas of Sullivan and Kreck, we define the bordism theory of $\ZZ_k$-stratifolds, and we can consider the generalized homology theory of bordism of $\ZZ_k$-stratifolds with continuous maps to $X$, denoted by $SH_*(X;\ZZ_k)$.
We call this theory {\bf $\ZZ_k$-stratifold homology}.
Again, we have a long exact sequence satisfying the commutative diagram 
\begin{equation}\label{diag2}
    \xymatrix{\cdots \ar[r] & SH_n(X) \ar[r]^{\times k}\ar[d]^h &  SH_n(X) \ar[r]^r \ar[d]^h & SH_n(X;\ZZ_k)\ar[r]^\delta\ar[d]^{h_{\ZZ_k}} & SH_{n-1}(X)\ar[r]\ar[d]^h &\cdots\\
    \cdots \ar[r] & H_n(X) \ar[r]^{\times k} &H_n(X)\ar[r]^r & H_n(X;\ZZ_k)\ar[r]  & H_{n-1}(X)\ar[r] &\cdots}
\end{equation}
In this case, the Hurewicz homomorphism $h_{\ZZ_k}: SH_*(X;\mathbb{Z}_k) \rightarrow H_*(X;\mathbb{Z}_k)$ is constructed in the same vein as in the theory of $\ZZ_k$-manifolds. 
We show that $\ZZ_k$-stratifold homology satisfies the Eilenberg--Steenrod axioms on $CW$-complexes, in particular we show that the Mayer--Vietoris sequence axiom holds by using a regularity argument for $\ZZ_k$-stratifolds \cite{kreck}.
The main result of this paper is the following:

\bteo
An isomorphism exists between $\ZZ_k$-stratifold homology theory and singular homology with $\ZZ_k$-coefficients. This isomorphism is valid for all CW-complexes and is compatible with the Bockstein homomorphisms. 
\eteo

F\"uhring \cite{fuehring} develops a smooth version of the Baas--Sullivan theory of manifolds with singularities that is applied to the Positive Scalar Curvature problem. 
In a way, stratifolds and $\mathbb{Z}_k$-stratifolds are another kind of  smooth versions of the Baas--Sullivan theory of manifolds with singularities.
One of the advantages of stratifolds and $\ZZ_k$-stratifolds is 
a very concrete description
of the filtration of the Atiyah--Hirzebruch Spectral Sequence (AHSS) for oriented bordism and $\ZZ_k$-bordism. This geometric description of the AHSS for $\ZZ$-coefficients was given by Tene \cite{tene}, and for $\ZZ_k$-coefficients has the following form:

\bteo
For $k$ an odd number, the filtration for the AHSS of $\ZZ_k$-bordism
\begin{equation}  E_{n,0}^{\infty} \subseteq \dots \subseteq E^{r+2}_{n,0} \subseteq \dots   \subseteq E_{n,0}^{2} \cong H_n(X ; \mathbb{Z}_k)=SH_n(X;\integer_k)\,, \end{equation}
coincides with the set of classes generated by singular $\ZZ_k$-stratifolds in $X$, where the singular part is of dimension at most $n-r-2$.
\eteo

A fascinating application is the existence of a homology class in $H_{2i}(B\ZZ_p;\ZZ_p)$, $i\geq 3$ and $p$ an odd prime number, that cannot be represented by a $\ZZ_p$-manifold.
This is similar to the counterexample of Thom  for the Steenrod problem \cite[Chapter III]{Novfra}, which we explain geometrically in the paper by the authors \cite{AST}.

We organize the article as follows: 
Section \ref{sec2} outlines some basic facts about $\ZZ_k$-manifolds studied by Morgan--Sullivan \cite{sullivan}.  
In Section \ref{sec3}, we briefly introduce the language of stratifolds from Kreck \cite{kreck,kreck1}. 
Section \ref{sec4} introduces the main theorems of this work, where we combine the theory of $\ZZ_k$-manifolds from Sullivan and the theory of stratifolds from Kreck. Then we define  $\ZZ_k$-stratifolds and develop the basic theory of these objects. We show that the usual properties of stratifolds still remain valid. We show that $\ZZ_k$-stratifold homology satisfies the Eilenberg--Steenrod axioms on $CW$-complexes. 
Section \ref{sec8} develops the existence of the fundamental class,
and we postpone the proof of the existence of the  Mayer--Vietoris sequence until Appendix \ref{apendice}. 
In Section \ref{secc1}, we apply the results of Tene \cite{tene} to give a geometric description of the Atiyah--Hirzebruch spectral sequence for $\ZZ_k$-bordism ($k$ an odd number).
In Section \ref{secGR}, we use this description to find homology classes with $\ZZ_k$-coefficients that cannot be represented by $\ZZ_k$-manifolds.
Finally, in Section \ref{secfin}, the two possible ways to represent homology with  $\mathbb{Z}_2$-coefficients using stratifolds are related, providing an explicit isomorphism between the two theories.

\subsection{Acknowledgements}
We thank the Math Institute UNAM-Oaxaca and Universidad de los Andes for the hospitality and financial support that made this collaboration possible. The first author acknowledges and thanks the hospitality and financial support provided by the Max Planck Institute for Mathematics in Bonn. This work was partially supported by the grant (\#INV-2019-84-1860) from the Fondo de Investigaciones de la Facultad de Ciencias de la Universidad de los Andes.
The second author is supported by c\'atedras CONACYT and Proyecto CONACYT Ciencias b\'asicas 2016, No. 284621. 
The third author's Ph.D. thesis \cite{torres} contains part of these results under the supervision of the first author. The maturity of the present paper is due to the guidance of the second author in two visits of the third author to the Math Institute UNAM-Oaxaca. Without the invaluable contribution of the second author, this work would not have been possible.
The third author would like to thank the Universidad Pontificia Javeriana for the help provided after his Ph.D. 
graduation and especially the Universidad Externado de Colombia, where he has been a professor in the mathematics department since 2020.  
Finally, we thank the anonymous Reviewer for the careful reading of our manuscript. We sincerely appreciate all the valuable comments and suggestions, which helped us improve the quality of the manuscript.

\section{$\ZZ_k$-Manifolds}
\label{sec2}
Suppose that $k\geq 2$ is a positive integer.
In what follows, we outline some basic facts about $\ZZ_k$-manifolds introduced by Morgan--Sullivan \cite{sullivan}.

\begin{noter}
For this note, unless otherwise indicated, let us set the convention that the manifolds are oriented and compact.
Also, all the diffeomorphisms and embeddings are orientation-preserving.
\end{noter}

\label{Zmani}
\bdefin
A closed $n$-dimensional $\ZZ_k$-{\bf manifold} is given by the triple $\mathcal{M}=(M,\delta M,\theta_i)$ where:  
\begin{itemize}
\item[(1)] $M$ is a compact $n$-manifold, with boundary $\partial M$,
\item[(2)] $\delta M$ is a compact $(n-1)$-manifold without boundary, called the Bockstein, and
\item[(3)] $\theta_i:\delta M\hookrightarrow\partial M$, with $i\in\ZZ_k$, are $k$ disjoint embeddings such that we have a diffeomorphism 
$\partial M=\bigsqcup_{i\in\ZZ_k}\theta_i(\delta M)$.
\end{itemize}
\edefin

\begin{definicion}
There is an associated {\bf quotient} space $\tilde{M}$ given by the identification on $M$ of the $k$ copies of $\delta M$ together using the embeddings $\theta_i$'s.
\end{definicion}

\begin{example} A closed oriented manifold is a  $\ZZ_0$-manifold (or equivalently a $\ZZ$-manifold) where the Bockstein $\delta M$ is empty.
\end{example}

\begin{example}
The typical example of a $\ZZ_2$-manifold is the cylinder $M=S^1\times [0,1]$, $\delta M=S^1$ and embeddings $\theta_1,\theta_2:S^1\hookrightarrow S^1\times \{0\}\sqcup S^1\times \{1\}$, with $\theta_1(S^1)=S^1\times \{0\}$ and $\theta_2(S^1)=S^1\times\{1\}$ ($S^1\times\{1\}$ with the reverse orientation). 
The quotient space $K:=\tilde{M}$ is the well-known Klein bottle; see Figure \ref{klein}.
\begin{figure}
    \centering
    \includegraphics[scale=0.2]{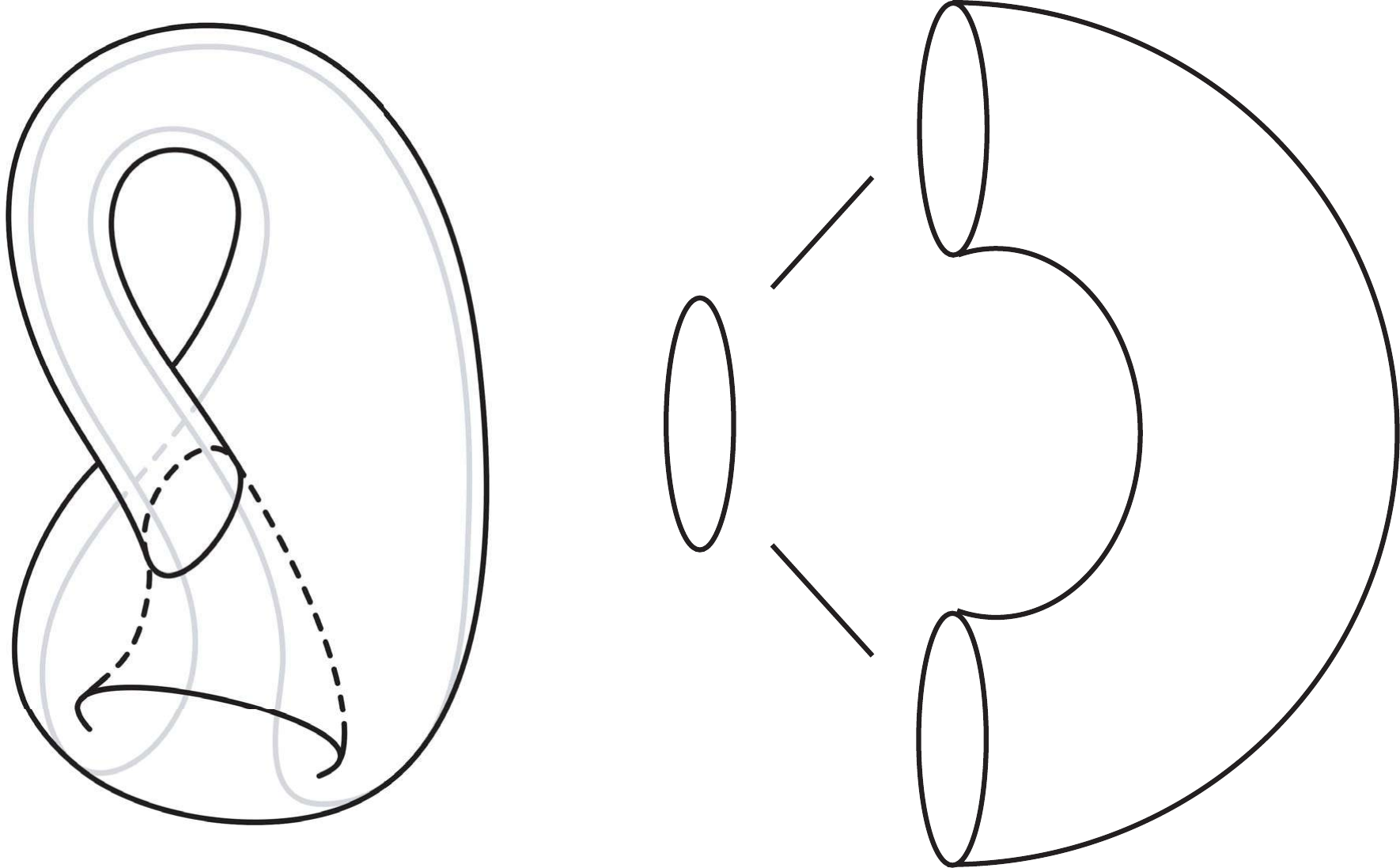}
    \caption{Representation of the Klein bottle as the quotient space of a $\ZZ_2$-manifold.}
    \label{klein}
\end{figure}
\end{example}

Here we observe that even though the second integral homology group is zero for the Klein bottle, we can obtain a fundamental class after we change to $\ZZ_2$ coefficients, i.e., $H_2(K;\ZZ_2)\cong\ZZ_2$.
In Section \ref{sec8}, we show this fundamental class always exists for a $\ZZ_k$-stratifold.

\bejem\label{Zkmani}
Consider the pair of pants $P$ with boundary $\partial P = S^1 \sqcup S^1 \sqcup S^1$ and Bockstein $\delta P=S^1$; see Figure \ref{z3mani}.
\begin{figure}
    \centering
    \includegraphics[scale=1]{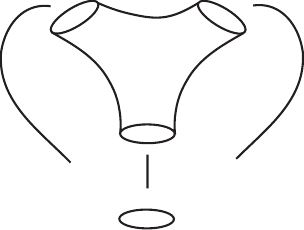}
    \caption{ A closed $\ZZ_3$-manifold.}
    \label{z3mani}
\end{figure}

\eejem

In what follows, we define the notion of a $\ZZ_k$-manifold with boundary.

\bdefin\label{defZkman}
An $(n+1)$-dimensional $\ZZ_k$-{\bf manifold with boundary} is given by the triple $\mathcal{B}=(B,\delta B,\psi_i)$ where:
\begin{itemize}
\item[(1)] $B$ is a compact $(n+1)$-dimensional manifold, with boundary $\partial B$,
\item[(2)] $\delta B$ is a compact $n$-dimensional manifold, called the Bockstein, with boundary $\partial\delta B$, and
\item[(3)] $\psi_i:\delta B\hookrightarrow \partial B $, with $i\in \ZZ_k$, are $k$ disjoint embeddings such that the triple 
$$ \left(\partial B-int( \bigsqcup_{i\in\ZZ_k} \psi_i(\delta B)),\partial \delta B,{\psi_i}|_{\partial \delta B}\right)$$ defines a closed $n$-dimensional $\ZZ_k$-manifold $(M,\delta M,\theta_i)$. 
\end{itemize}
This closed $n$-dimensional $\ZZ_k$-manifold is called the {\bf $\ZZ_k$-boundary} of the $\ZZ_k$-manifold with boundary $\mathcal{B}$ and is denoted by $\partial \mathcal{B}=(M,\delta M,\theta_i)$.
\edefin

\begin{definicion}
As before, there is the {\bf quotient} space $\tilde{B}$ which results by the identification on $B$ of the $k$ embedded copies of $\delta B$ together using the embeddings $\psi_i$'s.
\end{definicion}

\bejem
Consider the $3$-dimensional $\ZZ_3$-manifold with boundary $\mathcal{B}=(B,\delta B,\psi_i)$ where $B=D^3$ is the three dimensional closed ball (hence $\partial B=S^2$), $\delta B=D^2$ is the two-dimensional closed disc
and $\psi_i:D^2\longrightarrow S^2$, with $i\in \ZZ_3$, are given by three disjoint embedded discs inside the sphere. 
The $\ZZ_3$-boundary $\partial \mathcal{B}=(M,\delta M,\theta_i)$ is
the $2$-dimensional $\ZZ_3$-manifold from Example \ref{Zkmani}, where $M$ is the pair of pants and $\delta M$ is the circle. See Figure \ref{zkboundarym} for an illustration. 
\eejem

\begin{figure}
    \centering
    \includegraphics[scale=1]{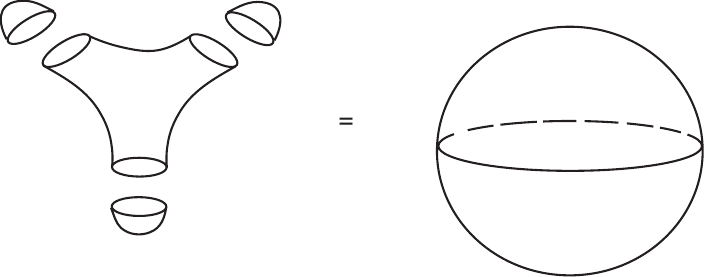}
\caption{A $\ZZ_3$-manifold with boundary.}
    \label{zkboundarym}
\end{figure}

\begin{example}Consider the $2$-dimensional $\ZZ_3$-manifold with boundary $\mathcal{B}=(B,\delta B,\psi_i)$ where $B$ is a connected surface of genus one with only one boundary circle, the Bockstein $\delta B$ is the interval $[0,1]$ and 
$\psi_i:[0,1]\rightarrow \partial B=S^1$, with $i\in \ZZ_3$, are given by three disjoint embedded intervals inside the circle. 
The $\ZZ_3$-boundary of the $\ZZ_3$-manifold $\mathcal{B}$ is a $1$-dimensional $\ZZ_3$-manifold $\partial \mathcal{B}=(M,\delta M,\theta_i)$ with $M$ the disjoint union of three copies of the interval, $\delta M$ is the disjoint union of two points and the embeddings $\theta_i$ are given by the restrictions $\psi_i|_{\delta M}$.
In Figure \ref{fig-1}, we illustrate the $\ZZ_3$-stratifold $(B,\delta B,\psi_i)$ where on the right side we depict the boundary $\partial B$ after the quotient.

\begin{figure}
\centering
\includegraphics[scale=1]{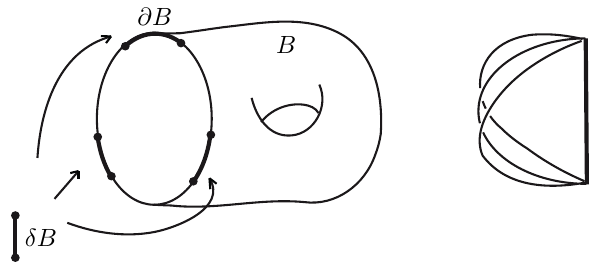}
\caption{Left: a $\ZZ_3$-manifold with boundary / Right: the boundary $\partial B$ after quotient.}
\label{fig-1}
\end{figure}
\end{example}

\begin{defn}
Let $X$ be a topological space and $n$ a natural number. 
An $n$-dimensional {\bf singular $\ZZ_k$-manifold} in $X$ is a closed $n$-dimensional $\ZZ_k$-manifold $\mathcal{M}=(M,\delta M,\theta_i)$ together with a continuous map $f:M\longrightarrow X$ such that
$f\circ\theta_i=f\circ\theta_j$ for $i,j\in \ZZ_k$.
A {\bf singular $\ZZ_k$-bordism} between two $n$-dimensional singular $\ZZ_k$-manifolds $(\mathcal{M},f)$ and $(\mathcal{M}',f')$ is a  $\ZZ_k$-manifold with boundary $\mathcal{B}=(B,\delta B,\psi_i)$, with $\ZZ_k$-boundary 
$\partial \mathcal{B}=(M+M',\delta M+\delta M',f+f')$ together with a continuous map $F:B\longrightarrow X$,
with $F\circ \psi_i=F\circ \psi_j$ for $i,j\in \ZZ_k$, extending $f$ and $f'$.
Recall that the $\ZZ_k$-manifolds are oriented. In this definition, the sum of $\ZZ_k$-manifolds is given by 
$$(M+M',\delta M+\delta M',f+f')=(M\sqcup -M',\delta M\sqcup -\delta M',f\sqcup f')\,.$$
The {\bf $\ZZ_k$-bordism group} group $\Omega_n(X;\ZZ_k)$ is given by the equivalence classes of $n$-dimensional singular $\ZZ_k$-manifolds $(\mathcal{M},f)$ under this $\ZZ_k$-bordism relation. The elements of this group are denoted by $[\mathcal{M},f]$.
\end{defn}

The $\ZZ_k$-bordism groups $\Omega_n(X;\ZZ_k)$ are a generalized homology theory (this follows by Section \ref{sec4} or see \cite[Chapter III]{AGA}). 
The existence of the fundamental class $[\mathcal{M}]_{\ZZ_k}\in H_n(\tilde{M};\ZZ_k)$, see Section \ref{sec8}, induces the Hurewicz homomorphism 
$h_{\ZZ_k}:\Omega_n(X;\ZZ_k)\rightarrow H_n(X;\ZZ_k)$. In addition, we have the reduction map $r:\Omega_n(X)\rightarrow \Omega_n(X;\ZZ_k)$. This map considers an $n$-dimensional closed manifold as a $\ZZ_k$-manifold with $\delta M=\emptyset$. 
Moreover, we have the Bockstein sequence, which fits into the following commutative diagram  
\begin{equation}\label{yiyik}
    \xymatrix{\cdots \ar[r] & \Omega_n(X) \ar[r]^{\times k}\ar[d]^h &  \Omega_n(X) \ar[r]^r \ar[d]^h & \Omega_n(X;\ZZ_k)\ar[r]^\delta\ar[d]^{h_{\ZZ_k}} & \Omega_{n-1}(X)\ar[r]\ar[d]^h &\cdots\\
    \cdots \ar[r] & H_n(X) \ar[r]^{\times k} &H_n(X)\ar[r]^r & H_n(X;\ZZ_k)\ar[r]  & H_{n-1}(X)\ar[r] &\cdots}
\end{equation}
for $n\geq 1$.

\section{Stratifolds}
\label{sec3}
We briefly introduce the language of stratifolds from Kreck \cite{kreck,kreck1}. For this purpose, we need the notion of differential space \cite{sirk,kreck,kreck1}.

\begin{definicion}
A \textbf{differential space} is a pair $(X,\mathcal{C})$ where $X$ is a topological Hausdorff space with a countable basis and $\mathcal{C} \subset C^0(X)$ is a sheaf of real-valued continuous functions, such that for $f_1,\cdots, f_k$ in $\mathcal{C}$ and $f$ a smooth function on $\RR^k$, the composition $f(f_1,\cdots,f_k)$ is in $\mathcal{C}$.% (this property is called ``locally detectable"). 
\end{definicion}

For a differential space, each point $x\in X$ has associated a tangent space, denoted by $T_x X$, which is the space of all derivations of the germ $\Gamma_x(\mathcal{C})$ of smooth functions at $x$. 
A smooth mani\-fold is a natural example of a differential space, which is locally diffeomorphic to $\RR^n$ equipped with the sheaf of all smooth functions.

\begin{defn}\label{stratifold}\cite[Def.~1]{kreck1}
An $n$-dimensional {\bf stratifold} is a differential space $(S,\mathcal{C})$ where the sheaf $\mathcal{C}$ induces a suitable stratification $S^k:=\{x\in S:\operatorname{dim} T_x S=k \}$. The union of all strata of dimension $\leq k$ is called the $k$-skeleton $S_k$. In addition, we assume:
\begin{itemize}
    \item[(i)] For each $k$ the stratum $S^k$, together with the restriction sheaf $\mathcal{C}|_{S^k}$, is a smooth $k$-dimensional manifold as differential space.
    \item[(ii)] All skeleta are closed subsets of $S$.
    \item[(iii)] All strata of dimension $> n$ are empty.
    \item[(iv)] For each $x\in S$ and open neighborhood $U$, with $x\in U$, there is a so-called bump function $\rho:S\longrightarrow\RR_{\geq 0}$ in $\ma{C}$, such that $\op{supp}\rho\subset U$ and $\rho(x)>0$.
    \item[(v)] \label{ev}For each $x\in S^k$ the restriction gives an isomorphism $\Gamma_x(\mathcal{C})\longrightarrow\Gamma_x(\mathcal{C}|_{S^k})$.
\end{itemize}
\end{defn}
\begin{defn}
A continuous map $f:(S,\mathcal{C})\longrightarrow (S',\mathcal{C}')$ is {\bf smooth}, if the precomposition by $f$ 
sends every element of $\mathcal{C}'$ to an element of $\mathcal{C}$. If $f$ and the inverse $f^{-1}$ are smooth, then $f$ is called a {\bf diffeomorphism of stratifolds}. 
Similarly, we can define the notion of a (smooth) {\bf embedding of stratifolds} by requiring that the restriction to the image is a diffeomorphism of stratifolds.
\end{defn}
\begin{example}(\cite[Ex.~1, p.~19]{kreck}) The open cone of an $n$-dimensional manifold $\stackrel{\circ}{CM} := M\times [0,1)/_{M\times\{0\}}$ is an example of an $(n+1)$-dimensional stratifold, where $\mathcal{C}$ consists of all continuous functions on $\stackrel{\circ}{CM}$ which are constant on some open neighborhood of the point produced by collapsing $M\times \{0\}$ and whose restriction to $M\times (0,1)$ is smooth. 
\end{example}

\begin{defn}
Let $W$ be a smooth manifold. A {\bf collar} is a homeomorphism $c:\partial W\times [0,\epsilon)\rightarrow U$, with $\epsilon>0$, where $U$ is an open neighborhood of $\partial W$ in $W$ such that 
$c|_{\partial W\times \{0\}}=\op{id}_{\partial W}$ and $c|_{\partial W\times (0,\epsilon)}$ is a diffeomorphism onto $U-\partial W$.
\end{defn}

\begin{definicion}
Let $(T,\partial T)$ be a pair of topological spaces. Assume $\stackrel{\circ}{T}=T-\partial T$ and $\partial T$ are stratifolds of dimension $n$ and $n-1$ with $\partial T\subset T$ a closed subspace. 
A {\bf collar} of $\partial T$ into $T$ is a homeomorphism $c:\partial T\times [0,\epsilon)\rightarrow U$, with $\epsilon>0$, where $U$ is an open neighborhood of $\partial T$ in $T$ such that 
$c|_{\partial T\times \{0\}}=\op{id}_{\partial T}$ and $c|_{\partial T\times (0,\epsilon)}$ is a diffeomorphism of stratifolds onto $U-\partial T$.
\end{definicion}

\begin{defn}
An $(n+1)$-dimensional {\bf stratifold with boundary} is a pair of topological spaces 
$(T,\partial T)$, together with a collar $c$ of $\partial T$ into $T$, where $T-\partial T$ is an $(n+1)$-dimensional stratifold and $\partial T$ is an $n$-dimensional stratifold, which is a closed subspace of $T$. We call $\partial T$ the {\bf boundary} of $T$.
\end{defn}

The following example is crucial in the theory of stratifolds.

\begin{example}\label{conoe}(\cite[p.~36]{kreck})
The closed cone $C(S)$ of a stratifold $S$ has underlying topological space $T=S\times [0,1]/_{S\times \{0\}}$ whose interior is $S\times [0,1)/_{S\times\{0\}}$ and whose boundary is $S\times \{1\}$. The collar is given by the map $S\times [0,1/2)\rightarrow C(S)$ mapping $(x,t)$ to $(x,1-t)$. %See Figure \ref{cone} for an illustration.
\end{example}

Now, we define some important classes of stratifolds \cite{kreck}.

\begin{definicion}(\cite[p.~79]{kreck})
 An $n$-dimensional stratifold $S$ is \textbf{oriented} if the top stratum $S^n$ is an oriented manifold and the stratum $S^{n-1}$ is empty. 
 \end{definicion}

  \begin{definicion}(\cite[p.~43]{kreck})
  An $n$-dimensional stratifold $S$ is \textbf{regular} if for each $x \in S^i$, $0\leq i\leq n$, there is an open neighborhood $U$ of $x$ in $S$, a stratifold $F$ with $F^0$ a single point, an open subset $V$ of $S^i$, and a diffeomorphism of stratifolds $\phi: V\times F \to U$, whose restriction to $V\times F^0$ is the identity.
  \end{definicion}

\bobs(\cite[p.~24]{kreck})\label{pstratifold}
In this note, we restrict to a special class of stratifolds called {\bf p-stratifolds}. 
The construction of a p-stratifold is as follows: we start with a $0$-dimensional p-stratifold, which is a $0$-dimensional manifold. Assume we construct by induction a $(k-1)$-dimensional p-stratifold $(S,\mathcal{C})$ and let $W$ be a $k$-dimensional manifold with a smooth and proper map $f:\partial W\longrightarrow S$. Then, we define the $k$-dimensional p-stratifold 
$(W\sqcup_f S,\mathcal{C}')$ where $\mathcal{C}'$ is constructed using a collar $c:\partial W\times [0,\epsilon)\rightarrow U$. More precisely, the function $g$ belongs to $\mathcal{C}'$ if and only if $g|_{S}$, $g|_{W-\partial W}$ are smooth and for some $\delta<\epsilon$ we have $gc(x,t) = gf(x)$ for all $x \in \partial W$ and $t < \delta$.  
\eobs

\begin{noter}
A stratifold with boundary $T$ is an oriented/regular stratifold, if both $T - \partial T$ and $\partial T$ are oriented/regular stratifolds (the collar preserves the product orientation for oriented stratifolds). Similarly, $T$ is a $p$-stratifold if both $T-\partial T $ and $\partial T$ are $p$-stratifolds.
\end{noter}

From Section \ref{sec4}, until the end of this note, all future statements about stratifolds are meant as statements about $p$-stratifolds (see Note \ref{notapstr}). 

As Kreck mentions in \cite[p.~303]{kreck1}: ``The following observation is central for our construction of the zoo of bordism groups".
For two stratifolds $T$ and $T'$ with the same boundary $\partial T=\partial T'$, there is a stratifold structure for the gluing of stratifolds $T\cup_{\partial T} T'$ where the two collars are combined to produce a {\bf bicollar}, see the details in \cite[p.~36-37]{kreck}.  

\begin{defn}
Let $X$ be a topological space and $n$ a natural number. 
An $n$-dimensional {\bf singular stratifold} in $X$ is a closed (compact without boundary) $n$-dimensional stratifold $S$ 
together with a continuous map $f:S\longrightarrow X$. 
A {\bf singular bordism} between two $n$-dimensional singular stratifolds $(S,f)$ and $(S',f')$ 
is a compact stratifold with boundary $T$, with boundary $(S+S',f+f')$ together with a continuous map $F:T\longrightarrow X$ extending $f$ and $f'$. The sum of oriented stratifolds 
is given by 
$$
(S+S',f+f')=(S\sqcup -S',f\sqcup f)\,.
$$
Since one can glue $n$-dimensional singular stratifolds over a common boundary component, singular bordism is an equivalence relation.
The {\bf oriented stratifold homology} group $SH_n(X)$ consists of the equivalence classes of $n$-dimensional oriented singular stratifolds $(S,f)$ under this bordism relation.
The elements of these groups are denoted by $[S,f]$.
\end{defn}

The significance of the previous bordism groups lies in the positive solution for the Steenrod problem \cite{Eil} with the aim that a geometric object represents integral homology classes. 
The precise statement is as follows:

\bteo (Kreck \cite[Thm.~20.1, p.~186]{kreck})
The functor $SH_*$ defines a homology theory. Moreover, there exists a natural transformation $h$, from $SH_*(\,\cdot\, )$ to singular homology $H_*(\cdot\, ;\mathbb{Z})$, such that $h$ is an isomorphism for all CW complexes. 
\eteo

\section{$\integer_k$-Stratifolds}
\label{sec4}

In this section we combine the theory of $\ZZ_k$-manifolds from Sullivan and the theory of stratifolds from Kreck.

\begin{noter}\label{notapstr}
For this note, unless otherwise indicated, let us set the conventions that the stratifolds are oriented, regular $p$-stratifolds. Also, all the diffeomorphisms and embeddings of stratifolds are orientation-preserving.
\end{noter}

\begin{definicion} \label{closeds}A closed $n$-dimensional $\integer_k$-{\bf stratifold} is given by the triple $\mathcal{S}=(S,\delta S,\theta_i)$ where: 
\begin{itemize}
\item[(1)] $S$ is a compact, $n$-dimensional stratifold, with boundary $\partial S$,
\item[(2)] $\delta S$ is a compact $(n-1)$-dimensional stratifold without boundary, called the Bockstein, and
\item[(3)] $\theta_i:\delta S \longrightarrow \partial S$, with $i\in \ZZ_k$, are $k$ disjoint embeddings of stratifolds such that we have a diffeomorphism of stratifolds $\partial S=\bigsqcup_{i\in\ZZ_k}\theta_i(\delta S)$.
\end{itemize}
\end{definicion}

\begin{definicion}
\label{remi1}
There is an associated {\bf quotient} space $\tilde{S}$ given by the identification on $S$ of the $k$ copies of $\delta S$ together using the embeddings $\theta_i$'s. 
\end{definicion}

\bejem The class of closed stratifolds and the class of $\ZZ_k$-manifolds are the first examples of $\ZZ_k$-stratifolds.
\eejem

\bejem\label{Zkstrati}
Consider the $2$-dimensional $\ZZ_3$-stratifold given by the 
closed cone of the disjoint union of three circles $S=C(S^1\sqcup S^1\sqcup S^1)$, where 
the boundary is $\partial S=S^1\sqcup S^1\sqcup S^1$, and the Bockstein is $\delta S=S^1$, see Figure \ref{z3srtatifold}.
\begin{figure}[ht]
    \centering
    \includegraphics[scale=1]{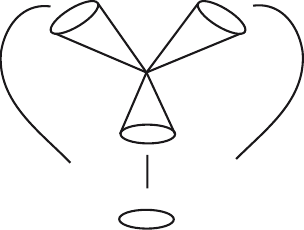}
    \caption{ A closed $\ZZ_3$-stratifold.}
    \label{z3srtatifold}
\end{figure}

\eejem

\begin{definicion}
    \label{frontera1} An $(n+1)$-dimensional $\integer_k$-{\bf stratifold with boundary} is given by the triple $\mathcal{T}=(T,\delta T,\psi_i)$ where:
        \begin{itemize}
            \item[(1)] $T$ is a compact $(n+1)$-dimensional stratifold, with boundary $\partial T$,
            \item[(2)] $\delta T$ is a compact $n$-dimensional stratifold with boundary, called the Bockstein, with boundary $\partial\delta T$, and
            \item[(3)] $\psi_i:\delta T \hookrightarrow \partial T$, with $i\in \ZZ_k$, are $k$ disjoint embeddings of stratifolds such that the triple 
            $$\left(\partial T- int(\bigsqcup_{i\in \ZZ_k} \psi_i(\delta T)), \partial \delta T, {\psi_i}|_{\partial \delta T}  \right)$$  
            defines a closed $n$-dimensional $\ZZ_k$-stratifold $(S,\delta S,\theta_i)$. 
         \end{itemize}  
         This closed $n$-dimensional $\ZZ_k$-stratifold is called the {\bf $\ZZ_k$-boundary} of the $\ZZ_k$-stratifold $\mathcal{T}$ and is denoted by $\partial \mathcal{T}=(S,\delta S,\theta_i)$. 
\end{definicion}

\begin{definicion}
There is a {\bf quotient} space $\tilde{T}$ resulting from the identification on $T$ of the $k$ copies of $\delta T$ together using the embeddings $\psi_i$'s.
\end{definicion}

\bejem
A $\integer_k$-manifold with boundary is an example of a $\integer_k$-stratifold with boundary.
\eejem

\bejem
Consider the $3$-dimensional $\ZZ_3$-stratifold with boundary $\mathcal{T}=(T,\delta T,\psi_i)$ where $T$ is the wedge of three closed balls $D^3\vee D^3\vee D^3$ by the north pole of the boundary spheres, hence the boundary is $\partial T=S^2\vee S^2\vee S^2$. The stratifold structure over the wedge point is given by the open cone of the disjoint union of three discs. The Bockstein is the two-dimensional closed disc $\delta T=D^2$ 
and $\psi_i:D^2\longrightarrow S^2\vee S^2\vee S^2$, with $i\in \ZZ_3$, are given by the embeddings of $D^2$ on each of the three 
southern hemispheres.  
The $\ZZ_3$-boundary $\partial \mathcal{T}=(S,\delta S,\theta_i)$ is
the $2$-dimensional $\ZZ_3$-stratifold from Example \ref{Zkstrati}, where $S=C(S^1\sqcup S^1\sqcup S^1)$ and the Bockstein is $\delta S=S^1$. See Figure \ref{zkboundary} for an illustration. 
\eejem

\begin{figure}[ht]
    \centering
    \includegraphics[scale=1]{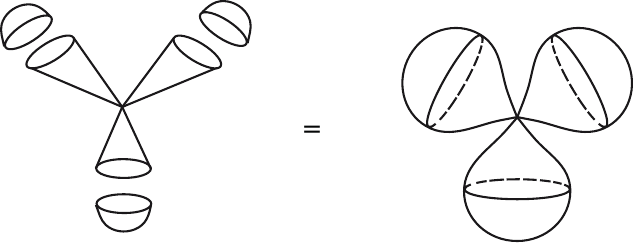}
\caption{A $\ZZ_3$-stratifold with boundary.}
    \label{zkboundary}
\end{figure}

\begin{definicion}\label{elcono}
The {\bf cone} of a $\ZZ_k$-stratifold $(S,\delta S,\theta_i)$ is defined as follows: take the closed cone $C(\delta S)$ (\cite[p.36]{kreck} or see Example \ref{conoe}) and use $k$-copies $kC(\delta S):=\bigsqcup_{i\in \ZZ_k}(C(\delta S)\times\{i\})$, to get the closed stratifold $S':=kC(\delta S)\sqcup_{\partial S} S$. Now, take the cone $C(S')$ which is an $(n+1)$-dimensional stratifold. The cone of the $\ZZ_k$-stratifold $(S,\delta S,\theta_i)$ is given by the $(n+1)$-dimensional $\ZZ_k$-stratifold with boundary $\mathcal{T}:=(C(S'),C(\delta S),\psi_i)$, where $\psi_i$ is the canonical inclusion in the $i$-component. The $\ZZ_k$-boundary of $\mathcal{T}$ is the original $\ZZ_k$-stratifold $(S,\delta S,\theta_i)$.
\end{definicion}    
        
       \begin{noter}
        For an $n$-dimensional $\ZZ_k$-stratifold $(S,\delta S, \theta_i)$, we need $n\geq 2$ in order to have that $C(S')$ and $C(\delta S)$ are oriented stratifolds. %Moreover, the $\ZZ_k$-stratifold 
       \end{noter}

The technique to show that the cartesian product of two differentiable manifolds has a differentiable structure is called {\bf straightening the angle}. We follow the exposition given by Conner-Floyd \cite[Sec.~I.3]{CF}.
Let $\RR_+\subset \RR$ consists of all non-negative real numbers. We have the homeomorphism $\tau:\RR_+\times \RR_+\rightarrow \RR\times \RR_+$, defined using polar coordinates by $\tau(\rho,\theta)=(\rho,2\theta)$, $0\leq \theta\leq \pi/2$, 
such that the restriction $\tau$ is a diffeomorphism of $\RR_+\times \RR_+\setminus (0,0)$ onto $\RR\times \RR_+\setminus (0,0)$. Consider the product of two differentiable manifolds $B_1$ and $B_2$ with collars $U_1$ and $U_2$ of the boundaries $\partial B_1$ and $\partial B_2$, respectively. There are diffeomorphisms $\Phi_1:U_1\rightarrow \partial B_1\times \RR_+$ and $\Phi_2:U_2\rightarrow\partial B_2\times \RR_+$. Let $U=U_1\times U_2$; then $\Phi=\Phi_1\times \Phi_2$ is a homeomorphism of $U$ onto $\partial B_1\times \partial B_2\times \RR_+\times \RR_+$ and the composition with $\tau'=\op{id}\times \tau$ produces a homeomorphism $\tau'\circ\Phi:U\rightarrow\partial B_1\times \partial B_2\times \RR\times\RR_+$. The differentiable structure of $\partial B_1\times \partial B_2\times \RR\times\RR_+$ induces a differentiable structure on $U$ such that $\tau'\circ \Phi$ is a diffeomorphism. Then $U$ and $B_1\times B_2\setminus \partial B_1\times \partial B_2$ have differentiable structures, and they induce the same differentiable structure on their intersection. This structure is referred to, as obtained by straightening the angle.

\bprop \label{product}
If $\mathcal{S}=(S,\delta S,\theta_i)$ is a closed $n$-dimensional $\integer_k$-stratifold, then after straightening the angle, we obtain an $(n+1)$-dimensional $\integer_k$-stratifold with boundary $\mathcal{S}\times [0,1]:=(S\times [0,1], \delta S\times [0,1],\psi_i)$ where the $\ZZ_k$-boundary $(S',\delta S',{\theta'}_i)$ is given by
\begin{itemize}
\item $S'=S\times \{0\}\sqcup -S\times \{1\}$,
\item $\delta S'=\delta S\times \{0\}\sqcup -\delta S\times \{1\}$ and
\item ${\theta'}_i=\theta_i\times \{0\}\sqcup\theta_i\times\{1\}$.
\end{itemize}
\eprop

\bdem
The technique of straightening the angle works similarly for the product of two stratifolds with boundary. In fact, from Kreck \cite[App.~A.1-A.2]{kreck}, we can use local retractions to show that the product of stratifolds has a stratifold structure.

Consequently, the product space $S\times [0,1]$ has the structure of compact $(n+1)$-dimensional stratifold with boundary, where $
\partial \left(S\times [0,1]\right)=(\partial S\times [0,1])\cup (S\times\{0,1\})$ is also a stratifold with a collar into
$S\times [0,1]$. Similarly, the product $\delta S\times [0,1]$ is a compact $n$-dimensional stratifold with boundary and we have embeddings $\theta_i\times \op{id}_{[0,1]}:\delta S\times [0,1]\hookrightarrow \partial S\times [0,1]$, with $i\in \ZZ_k$. Denote by $\psi_i$ the embedding obtained as the composition 
of $\theta_i\times\op{id}_{[0,1]}$ with the inclusion $\partial S\times [0,1]\hookrightarrow \partial \left(S\times [0,1]\right)$. We associate the $\ZZ_k$-stratifold with boundary $(T,\delta T,\psi_i)$ where $T:=S\times [0,1]$ and the Bockstein $\delta T:=\delta S\times [0,1]$.

From Definition \ref{frontera1}, it remains to show that the triple $(S',\delta S',\theta'_i):=(\partial T-\op{int}(\partial S\times [0,1]),\partial\delta T,\psi_i|_{\partial \delta T})$ is a closed $n$-dimensional $\ZZ_k$-stratifold.
We have, $S'=S\times \{0,1\}$, $\delta S'=\delta S\times \{0,1\}$ and the embeddings are ${\theta'}_i=\psi_i|_{\delta S'}=\theta_i\times \{0,1\}$. The orientation of $S\times [0,1]$ induces opposite orientations for the two copies of $S$ associated to $\{0,1\}$ and similarly for $\delta S$.
The embedding $\theta_i\times \{0\}$ preserves the orientation, while the embedding $\theta_i\times \{1\}$ reverses the orientation. This shows that $(S',\delta S',\theta'_i)$ is a $\ZZ_k$-stratifold which is the $\ZZ_k$-boundary of $\mathcal{S}\times [0,1]$.
\edem

Now we state a gluing lemma for $\ZZ_k$-stratifolds. This result is a direct application of Proposition A.1 in Kreck's book \cite[Prop.~A.1, p.~194]{kreck}.
 
\begin{lema}\label{gluingo}Let $\mathcal{T}:=(T,\delta T,\psi_i)$ and $\mathcal{T}':=(T',\delta T',\psi'_i)$ be $\ZZ_k$-stratifolds with $\ZZ_k$-boundaries $\partial \mathcal{T}=\mathcal{S}\sqcup \mathcal{S}'$ and $\partial \mathcal{T}'=\mathcal{S}\sqcup \mathcal{S}''$, where $\mathcal{S}=(S,\delta S,\theta_i)$, $\mathcal{S}'=(S',\delta S',\theta_i')$ and $\mathcal{S}''=(S'',\delta S'',\theta''_i)$, are closed $\ZZ_k$-stratifolds.  Then there is a $\ZZ_k$-stratifold with boundary $$\mathcal{T}\sqcup_{\mathcal{S}}\mathcal{T}':=(T\sqcup_{S} T',\delta T\sqcup_{\delta S} \delta T',\psi_i\sqcup_{\delta S}\psi_i')\,,$$
where the $\ZZ_k$-boundary is $\mathcal{S}'\sqcup \mathcal{S}''$.
\end{lema}

\begin{proof}
We consider the stratifolds $Y_1:=S'\sqcup_{\partial S'} \bigsqcup_{i\in \ZZ_k}\psi_i(\delta T)$ and $Y_2:=S''\sqcup_{\partial S''} \bigsqcup_{i\in \ZZ_k}\psi'_i(\delta T')$. Thus the boundary of the stratifold $T$ and $T'$ are $\partial T=S\sqcup_{\partial S}Y_1$ and $\partial T'=S\sqcup_{\partial S}Y_2$, respectively. 
The work of Kreck \cite[Prop.~A.1, p.~194]{kreck} implies that the gluing $T\sqcup_S T'$ is a stratifold with boundary, where $\partial (T\sqcup_S T')=Y_1\sqcup_{\partial S}Y_2$. Similarly, the gluing $\delta T\sqcup_{\delta S} \delta T'$ is a stratifold with boundary, which is the Bockstein. Thus the $\ZZ_k$-boundary is precisely $(S'\sqcup S'',\delta S'\sqcup \delta S'',\theta'_i\sqcup \theta''_i)$ and the lemma follows.
\end{proof}

\begin{defn}
Let $X$ be a topological space and $n$ a natural number. 
An $n$-dimensional {\bf singular $\ZZ_k$-stratifold} in $X$ is a closed $n$-dimensional $\ZZ_k$-stratifold $\mathcal{S}=(S,\delta S,\theta_i)$ together with a continuous map $f:S\longrightarrow X$ such that
$f\circ\theta_i=f\circ\theta_j$ for $i,j\in \ZZ_k$.
A {\bf singular $\ZZ_k$-bordism} between two $n$-dimensional singular $\ZZ_k$-stratifolds $(\mathcal{S},f)$ and $(\mathcal{S}',f')$ is a $\ZZ_k$-stratifold with boundary $\mathcal{T}=(T,\delta T,\psi_i)$, with $\ZZ_k$-boundary 
$\partial \mathcal{T}=(S+S',\delta S+\delta S',f+f')$ together with a continuous map $F:T\longrightarrow X$,
with $F\circ \psi_i=F\circ \psi_j$ for $i,j\in \ZZ_k$, extending $f$ and $f'$.
Recall that the $\ZZ_k$-stratifolds consist of oriented, regular $p$-stratifolds.
In this definition, the sum of $\ZZ_k$-stratifolds is given by 
$$(S+S',\delta S+\delta S',f+f')=(S\sqcup -S',\delta S\sqcup -\delta S',f\sqcup f')\,.$$
Again, one can glue $n$-dimensional singular $\ZZ_k$-stratifolds over a common boundary component. We state in Proposition \ref{equivalence} that singular $\ZZ_k$-bordism is an equivalence relation.
The {\bf $\ZZ_k$-stratifold homology} group $SH_n(X;\ZZ_k)$ is given by the equivalence classes of $n$-dimensional singular $\ZZ_k$-stratifolds $(\mathcal{S},f)$ under the $\ZZ_k$-stratifold bordism relation. We denote by $[\mathcal{S},f]$ the elements of this group.
\end{defn}

As a consequence of Proposition \ref{product} and the gluing Lemma \ref{gluingo}, we obtain the following.

\begin{proposicion}\label{equivalence}
The $\integer_k$-stratifold bordism relation is an equivalence relation.
\end{proposicion}
To any closed $n$-dimensional stratifold $S$, there is an associated  closed $n$-dimensional stratifold given by the disjoint union $kS:=\bigsqcup_{i\in \ZZ_k}S\times\{i\}$. This assignment produces the homomorphism 
\begin{equation}\label{0xk}
\times k:SH_n(X)\longrightarrow SH_n(X)\,.
\end{equation}
To any closed $n$-dimensional $\ZZ_k$-stratifold $\mathcal{S}=(S,\delta S,\theta_i)$, there is an associated  closed $n$-dimensional $\ZZ_k$-stratifold 
given by the disjoint union $kS:=\bigsqcup_{i\in \ZZ_k} S\times \{i\}$, where the Bockstein is the whole boundary $\partial S$ and the embeddings  $\psi_i:\partial S\rightarrow \bigsqcup_{i\in \ZZ_k} \partial S\times \{i\}$ are the canonical inclusions. 
Moreover, the boundary $\partial S=\bigsqcup_{i\in\ZZ_k}\theta_i(\delta S)$ can be considered as a $k$-disjoint union and we can denote $\left(kS,k\delta S,\psi_i\right):=\left(kS,\partial S,\psi_i\right)$.
This assignment produces the homomorphism 
\begin{equation}\label{1xk}
    \times k^k: SH_n(X;\ZZ_k)\longrightarrow SH_n(X;\ZZ_k)\,,
\end{equation}
which we show below that it is trivial.

\begin{proposicion}\label{p2} 
For every integer $n\geq 0$, the homomorphism $\times k^k:SH_n(X;\ZZ_k)\longrightarrow SH_n(X;\ZZ_k)$ is zero.
\end{proposicion}

\begin{proof}
Take $(\mathcal{S},f)=((S,\delta S),f)$ a closed singular $\ZZ_k$-stratifold.
Consider the stratifold with boundary given by the cylinder $T:=kS\times [0,1]$ and the Bockstein $\delta T:=(\partial S\times [0,1])\sqcup_{\partial S\times \{1\}}(-S\times \{1\})$ with embeddings $$\psi_i:\delta T\hookrightarrow \partial T=\left[(S\times \{0\})\sqcup_{\partial S\times \{0\}}(\partial S\times [0,1])\sqcup_{\partial S\times \{1\}}(-S\times \{1\})\right]\times\{i\}\,,$$
which are the canonical inclusions. The $\ZZ_k$-boundary of the $\ZZ_k$-stratifold $(T,\delta T,\psi_i)$ is the $k$-disjoint union of $(S,\delta S)$.
\end{proof}

Similar to the work of Morgan--Sullivan \cite{sullivan}, we have the Bockstein sequence, which fits into the following commutative diagram  
\begin{equation}\label{yiyi}
    \xymatrix{\cdots \ar[r] & SH_n(X) \ar[r]^{\times k}\ar[d]^h &  SH_n(X) \ar[r]^r \ar[d]^h & SH_n(X;\ZZ_k)\ar[r]^\delta\ar[d]^{h_{\ZZ_k}} & SH_{n-1}(X)\ar[r]\ar[d]^h &\cdots \ar[r]&SH_0(X;\ZZ_k)\ar[d] \\
    \cdots \ar[r] & H_n(X) \ar[r]^{\times k} &H_n(X)\ar[r]^r & H_n(X;\ZZ_k)\ar[r]  & H_{n-1}(X)\ar[r] &\cdots\ar[r]&H_0(X;\ZZ_k)\,.}
\end{equation}
The description of the maps is as follows:
\begin{itemize}
    \item the reduction $r:SH_n(X)\longrightarrow SH_n(X;\ZZ_k)$ is obtained by considering an $n$-dimensional closed stratifold as a $\ZZ_k$-stratifold, i.e., $(S,\delta S,\theta_i)$ with $\delta S=\emptyset$;
    \item multiplication $\times k:SH_n(X)\longrightarrow SH_n(X)$ takes a singular stratifold $(S,f)$ in $X$ and assigns the class of the $k$-disjoint union of $S$ denoted by $[kS,kf]$;
    \item the Bockstein $\delta:SH_n(X;\ZZ_k)\longrightarrow SH_{n-1}(X)$ assigns to a singular $\ZZ_k$-stratifold $(\mathcal{S},f)$, with $\mathcal{S}=(S,\delta S,\theta_i)$, the class $[\delta S,f|_{\delta S}]$;
    \item the Hurewicz homomorphism for stratifolds $h:SH_n(X)\longrightarrow H_n(X)$, with $n\geq 0$, was constructed by Kreck \cite[p.~186-187]{kreck};
    \item the Hurewicz homomorphism for $\ZZ_k$-stratifolds 
    $h_{\ZZ_k}:SH_n(X;\ZZ_k)\longrightarrow H_n(X;\ZZ_k)$, with $n\geq 0$,
    is constructed in Section \ref{sec8} where we show the existence of the fundamental class for $\ZZ_k$-stratifolds.
\end{itemize}
We leave the proof of the exactness of \eqref{yiyi} for Section \ref{sec7}, where the commutativity follows after we construct the fundamental class in Section \ref{sec8}.

Finally, we spend the rest of the section discussing the properties of $SH_*(\cdot;\ZZ_k)$ as a functor.
Kreck \cite{kreck} proves the Eilenberg--Steenrod axioms for the bordism groups $SH_*(\cdot)$ in the category of $CW$-complexes. We have a functor, i.e., $\op{id}_*=\op{id}$ and $(g\circ f)_*=g_*\circ f_*$, which is  homotopy invariant, has the Mayer--Vietoris sequence,  $SH_n(*)=0$ for $n\neq 0$ and $SH_0(*)=\ZZ$.  Similarly, the $\ZZ_k$-stratifold homology satisfies the Eilenberg--Steenrod axioms, that we show in detail below. 
The proof of the Mayer--Vietoris sequence is in Appendix \ref{MV}.

\bdefin
A continuous map $g:X \To Y$ defines a morphism between the $\ZZ_k$-stratifold bordism groups by 
\begin{align*}
    g_*: SH_n(X;\integer_k) & \To SH_n(Y;\integer_k)\,, \\
    [\mathcal{S},f] & \To [\mathcal{S},g\circ f]  
\end{align*}  
for $\mathcal{S}=(S,\delta S,\theta_i)$ a closed $n$-dimensional $\ZZ_k$-stratifold.
\edefin
This defines a functor which is  homotopy invariant, as in the following proposition.

\begin{proposicion}\label{homotopy}
If $g$ and $g'$ are homotopic maps from $X$ to $Y$,  then 
$$g_* = {g'}_* : SH_n(X;\integer_k) \To SH_n(Y;\integer_k)\,.$$
\end{proposicion}
\bdem
There is a homotopy $G: X \times [0,1] \To Y$ between $g$ and $g'$. Take $[\mathcal{S},f] \in SH_n(X;\integer_k)$, and hence 
$[\mathcal{S}\times [0,1],G\circ (f\times \op{id})]$
is a singular $\integer_k$-stratifold bordism (see Proposition \ref{product}) between $g_*([\mathcal{S},f])$ and ${g'}_*([\mathcal{S},f])$.
\edem

\begin{proposicion}\label{p1}
For the $\ZZ_k$-stratifold bordism group, we have
\begin{equation*}
    SH_n(*;\ZZ_k)=
    \left\{\begin{array}{cc}
    \ZZ_k & \textrm{ for }n=0 \\
     0    & \textrm{ for }n\neq 0\,.
\end{array} \right.
\end{equation*}
\end{proposicion}

\begin{proof}
An important assumption here is that every $n$-dimensional $\ZZ_k$-stratifold $(S,\delta S)$ is formed by oriented stratifolds $S$ and $\delta S$. For $n\geq 2$, we use the first horizontal long exact sequence of \eqref{yiyi}, with $SH_n(*)=0$ and $SH_{n-1}(*)=0$, and we conclude $SH_n(*;\ZZ_k)=0$. For $n=1$, the sequence \eqref{yiyi} becomes 
\begin{equation*}
    0\longrightarrow SH_1(*;\ZZ_k)\longrightarrow \ZZ\stackrel{\times k}{\longrightarrow}\ZZ\stackrel{r}{\longrightarrow}SH_0(*;\ZZ_k)\longrightarrow 0\,,
\end{equation*}
then $SH_1(*;\ZZ_k)=0$ and $SH_0(*;\ZZ_k)=\ZZ_k$. 
\end{proof}

A geometric approach for the previous proposition is as follows: for any closed $n$-dimensional $\ZZ_k$-stratifold $\mathcal{S}=(S,\delta S,\theta_i)$, with $n>1$, 
we take the cone as in Definition \ref{elcono}. Thus we consider the usual cone $C(\delta S)$ and use $k$-copies $kC(\delta S)$
to get the closed stratifold $S':=kC(\delta S)\sqcup_{\partial S} S$. Then we form the $(n+1)$-dimensional $\ZZ_k$-stratifold with boundary $\mathcal{T}:=(C(S'),C(\delta S),\psi_i)$ where $\psi_i$ is the canonical inclusion on the $i$th component. The $\ZZ_k$-boundary of $\mathcal{T}$ is the original $\ZZ_k$-stratifold $(S,\delta S,\theta_i)$.
For $n=1$, we have a disjoint union of circles and intervals with orientation. 
Since each interval has the boundary $\{+,-\}$, then the number of intervals must be divided by $k$. Thus after capping the circles with discs by Proposition \ref{p2}, this element is trivial in $SH_1(*;\ZZ_k)$.
Finally, for $n=0$, the generator of $SH_0(*;\ZZ_k)$ is the closed $0$-dimensional $\ZZ_k$-stratifold $(*,\emptyset,\op{id}_\emptyset)$, where we use Proposition \ref{p2}.

\section{The Bockstein sequence}
\label{sec7}

Previously, we have defined the $k$-disjoint union homomorphisms for stratifolds and $\ZZ_k$-stratifolds. These homomorphisms are as follows $\times k: SH_n(X)\longrightarrow SH_n(X)$ and $\times k^k:SH_n(X;\ZZ_k)\longrightarrow SH_n(X;\ZZ_k)$, defined in \eqref{0xk} and \eqref{1xk}, respectively. The second is the trivial homomorphism by Proposition \ref{p2}.
There is a third $k$-disjoint union homomorphism of the form 
\begin{equation}\label{2xk}\times k^{k^2}: SH_n(X;\ZZ_k)\longrightarrow SH_n(X;\ZZ_{k^2})\,,\end{equation}
which assigns to an $n$-dimensional $\ZZ_k$-stratifold $(S,\delta S)$ the $n$-dimensional $\ZZ_{k^2}$-stratifold 
$\left(kS,\delta S\right)$. 
There is a projection homomorphism $$p:SH_n(X;\ZZ_{k^2})\longrightarrow SH_n(X;\ZZ_k)\,,$$ which 
assigns to an $n$-dimensional $\ZZ_{k^2}$-stratifold $(S,\delta S)$ the $n$-dimensional $\ZZ_k$-stratifold $(S,k\delta S)$. 

We skip the embeddings and singular maps in defining these homomorphisms to simplify the notation.

 These homomorphisms satisfy a compatibility condition with the reduction and the Bockstein homomorphisms from the last section.

\begin{proposicion}\label{prim}
Let $r:SH_n(X)\longrightarrow SH_n(X;\ZZ_k)$ and $r:SH_n(X)\longrightarrow SH_n(X;\ZZ_{k^2})$ be the reduction homomorphisms and let $\delta :SH_n(X;\ZZ_{k^2})\longrightarrow SH_{n-1}(X)$ be the Bockstein homomorphism for $\ZZ_{k^2}$-stratifolds.
We have the following commutative diagrams:
$$
        \xymatrix{ SH_n(X;\ZZ_{k^2})\ar[r]^p\ar[d]^\delta& SH_n(X;\ZZ_k)\ar[d]^\delta\\ SH_{n-1}(X)\ar[r]^{\times k} & SH_{n-1}(X)\,,}
    \xymatrix{SH_n(X)\ar[r]^{\times k}\ar[d]^r & SH_n(X)\ar[d]^r\\ SH_n(X;\ZZ_k)\ar[r]^{\times k^k} & SH_n(X;\ZZ_k)\,,}
        \xymatrix{SH_n(X)\ar[r]^{\times k}\ar[d]^r & SH_n(X)\ar[d]^r\\ SH_n(X;\ZZ_k)\ar[r]^{\times k^{k^2}} & SH_n(X;\ZZ_{k^2})\,,}
$$    
and
    $$
    \xymatrix{SH_n(X)\ar[rd]^r\ar[d]^r&\\ SH_n(X;\ZZ_{k^2})\ar[r]^{p} & SH_n(X;\ZZ_k)\,,}
    \xymatrix{SH_n(X;\ZZ_{k^2})\ar[rd]^p&\\ SH_n(X;\ZZ_k)\ar[u]_{\times k^{k^2}}\ar[r]^{\times k^k} & SH_n(X;\ZZ_k)\,.}$$
\end{proposicion}
\begin{proof}
We show the commutativity of the first three squares. Take $(S,\delta S)$ an $n$-dimensional $\ZZ_{k^2}$-stratifold.
We have 
$k\delta S:=\times k(\delta S)=\times k\circ \delta (S,\delta S)$ and $k\delta S=\delta (S,k\delta S)=\delta \circ p(S,\delta S)$. 
Now, for $S$ a closed $n$-dimensional stratifold, we obtain 
$r\circ \times k(S)=(kS,\emptyset)$ and $\times k^k\circ r(S)=\times k^k(S,\emptyset)=(kS,\emptyset)$ in $SH_n(X;\ZZ_k)$. 
Similarly, we can show the commutativity of the third diagram with $(kS,\emptyset)$ in $SH_n(X;\ZZ_{k^2})$.
Finally, we show the commutativity of the last two diagrams. We have $r(S)=(S,\emptyset)=p(S,\emptyset)=p(r(S))$ and 
$p\circ \times k^{k^2}(S,\delta S)=p(kS,\delta S)=(kS,k\delta S)=\times k^k(S,\delta S)$. The commutativity of the second and fifth diagrams means that the composition is trivial by Proposition \ref{p2}.
\end{proof}

The following result shows how a  stratifold bordism gives rise to a  $\ZZ_k$-stratifold bordism.

\begin{proposicion}\label{solovino}
Assume that $\delta S$ and $\delta S'$ are two $n$-dimensional closed stratifolds such that there is a bordism of stratifolds $T$ with boundary $\partial T=\delta S\sqcup -\delta S'$. In addition, suppose the pair $(S,\delta S)$ is an $n$-dimensional $\ZZ_k$-stratifold. Then, 
$(S,\delta S)$ is $\ZZ_k$-bordant to $(S\sqcup_{\partial S}-kT,\delta S')$.
\end{proposicion}
\begin{proof}
This is similar to Proposition \ref{product}. We consider the product space $T':=\left(S\sqcup_{\partial S}-kT\right)\times [0,1]$ and the Bockstein $\delta T':=(\delta S'\times [0,1])\sqcup_{\delta S'\times \{1\}}-T$ with embeddings 
$$\psi_i:\delta T'\hookrightarrow \partial T'=\left(\left(S\sqcup_{\partial S}-kT\right)\times \{0\}\right)\sqcup_{k\delta S'\times \{0\}} k(\delta S'\times [0,1]) \sqcup_{k\delta S'\times \{1\}} \left(\left(S\sqcup_{\partial S}-kT\right)\times \{1\}\right)\,.$$
The $\ZZ_k$-stratifold $(T',\delta T',\psi_i)$ is a $\ZZ_k$-bordism between $(S,\delta S)$ and  
$(S\sqcup_{\partial S}-kT,\delta S')$.
\end{proof}

\begin{observacion}\label{obsdec}
Because of the relevance of the previous result for our work, in Figure \ref{ztst}, we illustrate two $\ZZ_k$-stratifolds that are $\ZZ_k$-bordant by  the previous proposition. 
\begin{figure}
    \centering
   \includegraphics[scale=0.7]{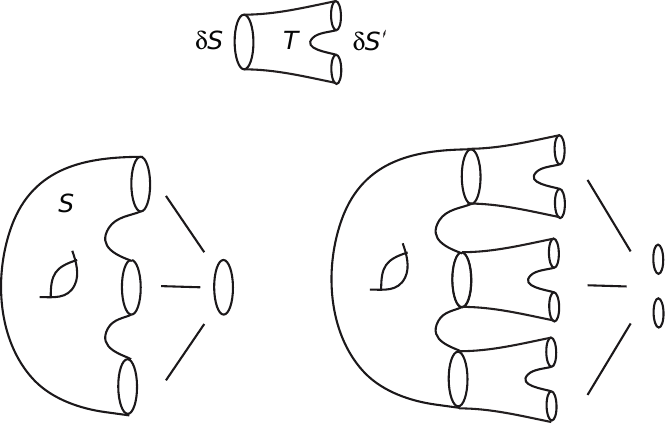}
    \caption{The bordism $T$ from $\delta S$ and $\delta S'$ and the two $\ZZ_k$-bordant $\ZZ_k$-stratifolds}
    \label{ztst}
\end{figure}
Notice that, whenever it is possible to connect $\delta S$ to the empty set by a bordism $T$, then the $\ZZ_k$-stratifold $(S,\delta S)$ is $\ZZ_k$-bordant to $(S\sqcup_{\partial S}-kT,\emptyset)$.
\end{observacion}

Similar to the work of Morgan--Sullivan \cite{sullivan}, the $\ZZ_k$-stratifolds bordisms groups have a Bockstein exact sequence associated with $0\rightarrow \ZZ\stackrel{\times k}{\rightarrow}\ZZ\rightarrow\ZZ_k\rightarrow 0$. There is also the other Bockstein exact sequence associated with $0\rightarrow\ZZ_k\stackrel{\times k}{\rightarrow}\ZZ_{k^2}\rightarrow\ZZ_k\rightarrow 0$. 
These two sequences are part of the commutative diagram below. 
The primary purpose of the present section is to show the exactness of the two Bockstein exact sequences. 

\begin{equation}\label{yiyi43}
    \xymatrix{\cdots \ar[r] & SH_n(X) \ar[r]^{\times k}\ar[d]^r &  SH_n(X) \ar[r]^r \ar[d]^r & SH_n(X;\ZZ_k)\ar[r]^\delta\ar[d]^{=} & SH_{n-1}(X)\ar[r]\ar[d]^r &\cdots\\
    \cdots \ar[r] & SH_n(X;\ZZ_k) \ar[r]^{\times k} &SH_n(X;\ZZ_{k^2})\ar[r]^p & SH_n(X;\ZZ_k)\ar[r]^{\tilde{\delta}}  & SH_{n-1}(X;\ZZ_k)\ar[r] &\cdots}
\end{equation}

\begin{proposicion}\label{Bock1}
The following sequence is exact 
$$\cdots \longrightarrow SH_n(X)\stackrel{\times k}{\longrightarrow} SH_n(X)\stackrel{r}{\longrightarrow}SH_n(X;\ZZ_k)\stackrel{\delta}{\longrightarrow}SH_{n-1}(X)\stackrel{\times k}{\longrightarrow}\cdots$$
\end{proposicion}

\begin{proof}
We have 
$r\circ (\times k)=(\times k^k)\circ r=0$ by Proposition \ref{prim}. In addition, 
we obtain $\delta\circ r=0$ since the Bockstein of a (closed) stratifold is empty. Moreover, $\times k\circ \delta =0$ since we start with a $\ZZ_k$-stratifold $(S,\delta S,\theta_i)$ where the boundary $\partial S$ is diffeomorphic to $\bigsqcup_{i\in \ZZ_k}\theta_i(\delta S)$. 

Now, we show  exactness: 
\begin{itemize}
    \item $\op{ker}r\subset \op{im}(\times k)$. Consider an $n$-dimensional singular stratifold $(S,f)$ with $r([S,f])=0$. Then  there is 
    an $(n+1)$-dimensional $\ZZ_k$-bordism $(\mathcal{T},F)=((T,\delta T),F)$ such that the $\ZZ_k$-boundary $\partial (T,\delta T)=(S,\emptyset)$ and $F$ extends $f$. Consequently, we obtain $\partial \delta T=\delta S=\emptyset$ and hence $\partial T=S \sqcup k\delta T$, and we can take the singular stratifolds given by $(\delta T,F|_{\partial T})$ with the reverse orientation. We have $k[-\delta T,-F|_{\partial T}]=[S,f]$. 
    \item $\op{ker}\delta\subset \op{im}r$. Consider an $n$-dimensional singular $\ZZ_k$-stratifold $(\mathcal{S},f)=((S,\delta S),f)$ with $\delta ([\mathcal{S},f])=0$. Then $(\delta S,f|_{\delta S})$ is the boundary of an $n$-dimensional singular stratifold $(T,F)$, i.e., $\partial T=\delta S$ and $F$ extends $f|_{\partial S}$. 
    Proposition \ref{solovino} and Remark \ref{obsdec}, imply that the $\ZZ_k$-stratifold $(S\sqcup_{\partial S} -kT,\emptyset)$ is $\ZZ_k$-bordant to $\ZZ_k$-stratifold $(S,\delta S)$.
   There is a map $f':S\sqcup_{\partial S} -kT\rightarrow X$ which extends the singular map $f$. Therefore, the singular $\ZZ_k$-stratifold $((S\sqcup_{\partial S} -kT,\emptyset),f')$ is $\ZZ_k$-bordant to the original singular $\ZZ_k$-stratifold $((S,\delta S),f)$.
\item $\op{ker}(\times k)\subset \op{im}\delta$. Consider an $(n-1)$-dimensional singular stratifold $(S,f)$ with $\times k([S,f])=0$. Then there exists an $n$-dimensional singular stratifold $(T,F)$ with $\partial T=k S$ and $F$ extends $kf$.
Thus we can take the $n$-dimensional singular $\ZZ_k$-stratifold $((T,S),F)$ and we obtain $\delta([(T,S),F])=[S,f]$.

\end{itemize}

\end{proof}
Denote by $\tilde{\delta}$ the following composition $SH_n(X;\ZZ_k)\stackrel{\delta}{\rightarrow} SH_{n-1}(X)\stackrel{r}{\rightarrow}SH_{n-1}(X;\ZZ_k)$.

\begin{proposicion}\label{Bock2}
The following sequence is exact 
$$\cdots \longrightarrow SH_n(X;\ZZ_k)\stackrel{\times k^{k^2}}{\longrightarrow} SH_n(X;\ZZ_{k^2})\stackrel{p}{\longrightarrow}SH_n(X;\ZZ_{k})\stackrel{\tilde{\delta}}{\longrightarrow}SH_{n-1}(X;\ZZ_{k})\stackrel{\times k^{k^2}}{\longrightarrow}\cdots$$
\end{proposicion}

\begin{proof}
We have  $p\circ (\times k^{k^2})=\times k^k= 0$ by Proposition \ref{prim}. Again we use Proposition \ref{prim}, and we get 
$$\tilde{\delta}\circ p= r\circ \delta\circ p =(r\circ (\times k))\circ \delta=0\,.$$
Similarly, we obtain 
$$(\times k^{k^2})\circ \tilde{\delta}= (\times k^{k^2})\circ r\circ \delta =(r\circ (\times k))\circ \delta=0\,.$$

Now, we show exactness:
\begin{itemize}
    \item $\op{ker} p\subset \op{im}(\times k^{k^2})$. Consider an $n$-dimensional singular $\ZZ_{k^2}$-stratifold $(\mathcal{S},f)=((S,\delta S),f)$ with $p([\mathcal{S},f])=0$. Then there exists an $(n+1)$-dimensional singular $\ZZ_k$-stratifold with boundary $(\mathcal{T},F)=((T,\delta T),F)$ such that the $\ZZ_k$-boundary is $\partial \mathcal{T}=(S,k\delta S)$. Thus we can consider $k$-copies of $\delta T$ with the reverse orientation, which are glued with $S$ to form a closed
    stratifold $S\sqcup_{\partial S} -k\delta T$ which is the boundary of $T$. 
    There are $k$ disjoint embeddings $c_i:\delta S\times [0,\epsilon)\hookrightarrow \delta T$ induced by the collar of $\partial S$ into the $k$ copies of $\delta T$. Denote by 
    $\overline{\delta T}:=\delta T-\bigsqcup_{i\in\ZZ_k}c_i(\delta S\times [0,\epsilon/2])$.
    We consider the $\ZZ_{k^2}$-stratifold with boundary $(T,\delta S\times [0,\epsilon/2],\psi_i)$, where $\psi_i=c_i|_{\delta S\times [0,\epsilon/2]}$.
 This is a $\ZZ_{k^2}$-bordism between $(S,\delta S)$ and $(k\overline{\delta T},\delta S)$. This means that $(\times k^{k^2})([\overline{\delta T},\delta S])=[k\overline{\delta T},\delta S]=[S,\delta S]$.
    
    \item $\op{ker} \tilde{\delta}\subset \op{im}p$. Consider an $n$-dimensional singular $\ZZ_k$-stratifold $(\mathcal{S},f)=((S,\delta S),f)$ with $\tilde{\delta}([\mathcal{S},f])=0$. Since $\tilde{\delta}=r\circ \delta $, this means that there exists an $n$-dimensional singular $\ZZ_k$-bordism $(\mathcal{T},F)=((T,\delta T),F)$ such that the $\ZZ_k$-boundary is $((\delta S,\emptyset),f|_{\delta S})$. Therefore, we obtain 
    $\partial T=\delta S\sqcup k\delta T$, $F$ extends $f|_{\delta S}$ and $\partial\delta T=\emptyset$. Consequently, we consider $k$-copies of $T$ with the reverse orientation, glued with $S$ to form the $n$-dimensional stratifold with boundary $S'=-k T\sqcup_{\partial S} S$. There is a map $f':S'\rightarrow X$ also constructed by the gluing.
    Thus we have an $n$-dimensional singular $\ZZ_{k^2}$-stratifold $((S',\delta T),f')$. We have $p([(S',\delta T),f'])=[(S',k\delta T),f']$ which is equal to $(\mathcal{S},f)$ by Proposition \ref{solovino}.

    \item $\op{ker} (\times k^{k^2})\subset \op{im}(\tilde{\delta})$. Consider an $(n-1)$-dimensional singular $\ZZ_k$-stratifold $(\mathcal{S},f)=((S,\delta S,\theta_i),f)$ with $\times k^{k^2}([\mathcal{S},f])=0$. Then there is an $n$-dimensional singular $\ZZ_{k^2}$-stratifold $(\mathcal{T},F)=((T,\delta T,\psi_i),F)$ with $\ZZ_{k^2}$-boundary $((kS,\delta S),kf)$. Therefore, $\partial T=kS\sqcup_{\partial kS}-k^2\delta T$ 
    is a closed $n$-dimensional stratifold. By the definition of the $\ZZ_{k^2}$-boundary of a $\ZZ_{k^2}$-stratifold with boundary, see Definition \ref{frontera1}, hence $\delta S=\partial \delta T$ and the embeddings are $\theta_i=\psi_i|_{\partial \delta T}$. Therefore, the gluing $S\sqcup_{\partial S} k\delta T$ is a closed $n$-dimensional stratifold and 
    in addition, we obtain 
    $\partial T$ is the disjoint union of $k$-copies of $S\sqcup_{\partial S} k\delta T$. 
    Consequently, we take the $(n+1)$-dimensional singular $\ZZ_k$-stratifold 
    $((T,S\sqcup_{\partial S} k\delta T),F)$ and $\tilde{\delta}([(T,S\sqcup_{\partial S} k\delta T),F])=[(S\sqcup_{\partial S} k\delta T,\emptyset),F|_{S\sqcup_{\partial S} k\delta T}]$ which is $\ZZ_k$-bordant to $((S,\delta S,\theta_i),f)$ by Proposition \ref{solovino}.
    \end{itemize}

\end{proof}

\section{Fundamental classes}
Recall from Section \ref{sec2} that a closed $\ZZ_k$-manifold $(M,\delta M,\theta_i)$ has an associated quotient space $\tilde{M}$. 
Similarly, we write $\tilde{\partial M}$ to mean the quotient space given by the identification on $\partial M$ of the $k$ copies of $\delta M$. Notice that in this case, we have $\tilde{\partial M}\cong\delta M$. 
Similarly, for a $\ZZ_k$-manifold with boundary $(B,\delta B,\psi_i)$, we denote by $\tilde{B}$ and $\tilde{\partial B}$ the quotient space obtained by the identification of the $k$ copies of $\delta B$ on $B$ and $\partial B$, respectively.

\label{sec8}
In this section, we will construct a natural transformation from $\ZZ_k$-bordism stratifold homology to homology with $\ZZ_k$-coefficients
\begin{equation}
    \Phi:SH_*(X;\ZZ_k)\longrightarrow H_*(X;\ZZ_k)\,.
\end{equation}

We can define this natural transformation for $\ZZ_k$-manifolds \cite{sullivan}. There is no formal proof of this fact in the literature, so we provide a detailed argument below.
The case of $\ZZ_k$-stratifolds uses some results of Tene \cite{tene2}. We give the details of these statements at the end of this section.

Assume $\mathcal{M}=(M,\partial M,\theta_i)$ is a closed $n$-dimensional $\ZZ_k$-manifold and there is a continuous map $f:M\rightarrow X$ to the topological space $X$. 
There exists the {\bf fundamental class} $[\mathcal{M}]_{\ZZ_k}\in H_n(\tilde{M};\ZZ_k)$
and for an element $[\mathcal{M},f]\in \Omega_n(X;\ZZ_k)$, there is a natural transformation  defined by 
\begin{equation}\label{we1}
    \Phi([\mathcal{M},f])=\tilde{f}_*([\mathcal{M}]_{\ZZ_k})\,,
\end{equation}
where $\tilde{f}:\tilde{M}\rightarrow X$ is the induced map from the quotient space $\tilde{M}$.

We can find the fundamental class $[\mathcal{M}]_{\ZZ_k}$ using the following commutative diagram:
\begin{equation}\label{eqod1}\xymatrix{
 \cdots \ar[r] & H_{n}(\partial M;\ZZ_k) \ar[r]\ar[d]^{q_*}    & H_{n}(M;\ZZ_k) \ar[d]^{q_*}\ar[r]^{i_*}    & H_{n}(M,\partial M;\ZZ_k) \ar[d]^{q_*}\ar[r]^{\partial}    & H_{n-1}(\partial M;\ZZ_k) \ar[d]^{q_*}\ar[r]  & \cdots \\
\cdots \ar[r] & H_{n}(\tilde{\partial M};\ZZ_k) \ar[r]    & H_{n}(\tilde{M};\ZZ_k) \ar[r]^{i_*}   & H_{n}(\tilde{M}, \tilde{\partial M};\ZZ_k) \ar[r]^{\partial} & H_{n-1}(\tilde{\partial M};\ZZ_k) \ar[r]        & \cdots
}\end{equation}
In the previous diagram, the rows are the long exact sequences associated with the pairs $(M,\partial M)$ and $(\tilde{M},\tilde{\partial M})$. 
The quotient map induces the vertical morphisms. We start with the well known fundamental class $[M,\partial M]\in H_n(M,\partial M;\ZZ_k)$ which satisfies $\partial([M,\partial M])=[\partial M]$
and 
\begin{equation}\begin{array}{rll}
    H_{n-1}(\tilde{\partial M};\ZZ_k)&\stackrel{\cong}{\longrightarrow} &H_{n-1}(\delta M;\ZZ_k) \,.\\
    q_*([\partial M]) &\longmapsto &k[\delta M] 
\end{array}
\end{equation}
Thus $q_*([\partial M])=0$ by the coefficients. We have the isomorphism $q_*:H_n(M,\partial M;\ZZ_k)\rightarrow H_n(\tilde{M}, \tilde{\partial M};\ZZ_k)$ and  
$H_n(\tilde{\partial M};\ZZ_k)\cong H_n\left(\delta M;\ZZ_k\right)=0$. 
Therefore, there exists a unique class $[\mathcal{M}]_{\ZZ_k}\in H_n(\tilde{M};\ZZ_k)$ with the property  \begin{equation}\label{cosa}i_*([\mathcal{M}]_{\ZZ_k})=q_*([M,\partial M])\,.\end{equation} 

The following lemma is needed to show the existence of relative fundamental classes for $\ZZ_k$-manifolds.

\begin{lema}\label{gato}
Let $M$ be a closed compact oriented manifold of dimension $n$. Assume $M$ is the gluing of two compact oriented manifolds with boundary of dimension $n$, i.e., \begin{equation}M=M_1\sqcup_{\partial M_1=\partial M_2}M_2\,.\end{equation} Then the composition 
    $H_n(M)\stackrel{i_*}{\longrightarrow} H_n(M,M_1)\stackrel{\cong}{\longrightarrow} H_n(M_2,\partial M_2)$
sends the fundamental class $[M]\in H_n(M)$ to the relative fundamental class $[M_2,\partial M_2]\in H_n(M_2,\partial M_2)$, where the isomorphism $H_n(M,M_1)\stackrel{\cong}{\longrightarrow} H_n(M_2,\partial M_2)$ is provided by excision.
\end{lema}

\begin{proof}
We have the commutative diagram 
\begin{equation}
\xymatrix{ & H_n(M_2,\partial M_2)\ar[d]^{exc}\ar[r] & H_n(M_2,M_2-\{x\})\ar[d]^{\cong}\\
H_n(M)\ar[r] & H_n(M,M_1) \ar[r] & H_n(M,M-\{x\})\,.}
\end{equation}
where $x\in  \stackrel{\circ}{M}_2=M_2-\partial M_2$. By classic algebraic topology \cite[Lemma~3.27]{hatcher}, the two rows send the fundamental classes to the generators associated with the point $x$, which shows the lemma.
\end{proof}

Now we show the existence of a {\bf relative fundamental class} of an $(n+1)$-dimensional $\ZZ_k$-manifold with boundary 
$\mathcal{B}=(B,\partial B,\psi_i)$ where the $\ZZ_k$-boundary is $\partial \mathcal{B}=(M,\delta M,\theta_i)$.
We find the fundamental class $[\mathcal{B},\partial\mathcal{B}]_{\ZZ_k}$ using the following commutative diagram:

 \begin{equation}\label{eq132}\xymatrix{
 \cdots \ar[r] & H_{n+1}(\partial B,M;\ZZ_k) \ar[r]\ar[d]^{q_*}    & H_{n+1}(B,M;\ZZ_k) \ar[d]^{q_*}\ar[r]^{i_*}    & H_{n+1}(B,\partial B;\ZZ_k) \ar[d]^{q_*}\ar[r]^{\partial}    & H_{n}(\partial B,M;\ZZ_k) \ar[d]^{q_*}\ar[r]  & \cdots \\
\cdots \ar[r] & H_{n+1}(\tilde{\partial B},\tilde{M};\ZZ_k) \ar[r]    & H_{n+1}(\tilde{B},\tilde{M};\ZZ_k) \ar[r]^{i_*}   & H_{n+1}(\tilde{B}, \tilde{\partial B};\ZZ_k) \ar[r]^{\partial} & H_{n}(\tilde{\partial B},\tilde{M};\ZZ_k) \ar[r]        & \cdots
}\end{equation}
In the previous diagram, the rows are the long exact sequences associated with the triples $(B,\partial B, M)$ and $(\tilde{B},\tilde{\partial B}, \tilde{M})$, respectively, and the quotient map induces the vertical morphisms. 
We start with the relative fundamental class $[B,\partial B]\in H_n(B,\partial B;\ZZ_k)$ and using Lemma \ref{gato} we have 
$\partial [B,\partial B]=[k\delta B,\partial M]$, where $k\delta B:=\bigsqcup_{i\in \ZZ_k}\psi_i(\delta B)$, and 
\begin{equation}
\begin{array}{rll}
  H_{n}(\tilde{\partial B},\tilde{M};\ZZ_k)& \stackrel{\cong}{\longrightarrow}  & H_{n}(\delta B,\delta M;\ZZ_k) \,.\\
  q_*[k\delta B, \partial M]& \longmapsto  & k[\delta B, \delta M]
\end{array}
\end{equation} 
Thus $q_*[k\delta B, \partial M]=0$ by the coefficients. We have an isomorphism $q_*:H_{n+1}(B,\partial B;\ZZ_k)\rightarrow H_{n+1}(\tilde{B}, \tilde{\partial B};\ZZ_k)$ and 
$H_{n+1}(\tilde{\partial B},\tilde{M};\ZZ_k)\cong H_{n+1}\left(\delta B,\delta M;\ZZ_k\right)=0$. Therefore, there exists a unique  
class $[\mathcal{B},\partial \mathcal{B}]_{\ZZ_k}\in H_{n+1}(\tilde{B},\tilde{M};\ZZ_k)$ with the property  \begin{equation}i_*([\mathcal{B},\partial\mathcal{B}]_{\ZZ_k})=q_*([B,\partial B])\,.\end{equation}

\begin{proposicion}
Let $\mathcal{B}=(B,\partial B,\psi_i)$ be an $(n+1)$-dimensional $\ZZ_k$-manifold with boundary, where the $\ZZ_k$-boundary is $\partial \mathcal{B}=(M,\delta M,\theta_i)$. Then, the image of $[\mathcal{B}, \partial\mathcal{B}]_{\ZZ_k}$ under the map $\partial: H_{n+1}(\tilde{B},\tilde{M};\ZZ_k)\longrightarrow H_n(\tilde{M};\ZZ_k)$ is the class $[\partial\mathcal{B}]_{\ZZ_k}$.
\end{proposicion}

\begin{proof}
We apply the differential maps to the middle square in \eqref{eq132} and we obtain the following commutative cube: 
\begin{equation}\label{njor}
\begin{tikzcd}[row sep=2.5em]
H_{n+1}(B,M;\ZZ_k)  \arrow[rr,"i_*"] \arrow[dr,swap,"\partial"] \arrow[dd,swap,"q_*"] &&
  H_{n+1}(B,\partial B;\ZZ_k) \arrow[dd,swap,"q_*" near start] \arrow[dr,"\partial"] \\
& H_{n}(M;\ZZ_k) \arrow[rr,crossing over,"i_*" near start] &&
  H_{n}(\partial B;\ZZ_k) \arrow[dd,"q_*"] \\
H_{n+1}(\tilde{B},\tilde{M};\ZZ_k) \arrow[rr,"i_*" near end] \arrow[dr,swap,"\partial"] && H_{n+1}(\tilde{B}, \tilde{\partial B};\ZZ_k) \arrow[dr,swap,"\partial"] \\
& H_{n}(\tilde{M};\ZZ_k) \arrow[rr,"i_*"] \arrow[uu,<-,crossing over,"q_*" near end]&& H_{n}(\tilde{\partial B};\ZZ_k)\,.
\end{tikzcd}
\end{equation}
We continue with the long exact sequence of the pairs $(\partial B,k\delta B)$ and $(\tilde{\partial B},k\tilde{\partial B})$ for the front square of \eqref{njor}, and we obtain the middle square in the following commutative diagram 
\begin{equation}\label{eq23pepe}\xymatrix{
 H_{n}(M;\ZZ_k) \ar[d]^{q_*}\ar[r]^{i_*}    & H_{n}(\partial B;\ZZ_k)  \ar[d]^{q_*} \ar[r]^{j_*} & H_{n}(\partial B, k \delta B;\ZZ_k)\ar[d]^{q_*} \ar[r]^{\cong}_{exc}& \ar[d]^{q*}_{\cong}  H_n(M, \partial M; \ZZ_k)\\
H_{n}(\tilde{M};\ZZ_k) \ar[r]^{i_*}   & H_{n}(\tilde{\partial B};\ZZ_k) \ar[r]^{j*} & H_{n}(\tilde{\partial B}, \tilde{k \delta B};\ZZ_k) \ar[r]^{\cong}_{exc} & H_n(\tilde{M}, \tilde{\partial M}; \ZZ_k)\,.
}\end{equation}
In the previous commutative diagram, we use excision for the third square on the right. Notice that the composition of the horizontal maps in \eqref{eq23pepe} are the maps 
$i_*:H_n(M;\ZZ_k)\longrightarrow H_n(M,\partial M;\ZZ_k)$ and $i_*:H_n(\tilde{M};\ZZ_k)\longrightarrow H_n(\tilde{M},\partial\tilde{M};\ZZ_k)$.

We chase the class $[\mathcal{B},\partial \mathcal{B}]_{\ZZ_k}\in H_{n+1}(\tilde{B},\tilde{M};\ZZ_k)$ in the diagrams \eqref{njor} and \eqref{eq23pepe}, where we obtain the following consequences 
\begin{align*}
    i_*\partial ([\mathcal{B},\partial \mathcal{B}]_{\ZZ_k})&= \partial i_*([\mathcal{B},\partial \mathcal{B}]_{\ZZ_k})\\
                                         &= \partial q_*([B,\partial B])\\
                                         &= q_*\partial([B,\partial B])=q_*([\partial B])\,.
\end{align*}
By Lemma \ref{gato}, we have the equation $j_*([\partial B])=[M,\partial M]$. Thus we obtain the property \eqref{cosa}
and the result follows.
\end{proof}

\begin{proposicion}\label{pancha}
The natural transformation $\Phi:\Omega_*(X;\ZZ_k)\longrightarrow H_*(X;\ZZ_k)$ is well defined.
\end{proposicion}
\begin{proof}
For an $n$-dimensional singular $\ZZ_k$-manifold $(\mathcal{M},f)$ which is null $\ZZ_k$-bordant, there exists an $(n+1)$-dimensional $\ZZ_k$-bordism $(\mathcal{B},F)$ with $\partial \mathcal{B}=\mathcal{M}$ where $F$ extends $f$. We have the commutative diagram:

\begin{equation}\label{eq23p}\xymatrix{
   [\mathcal{B}, \partial\mathcal{B}]_{\ZZ_k}  \in H_{n}(\tilde{B}, \tilde{M};\ZZ_k) \ar@<4ex>[d]^{\partial}    \ar[r] & \ar[d]^{\partial} H_{n}(X,X;\ZZ_k) = 0 \\
 [\mathcal{M}]_{\ZZ_k} \in  H_{n}(\tilde{M};\ZZ_k)  \ar[r] & H_{n}(X;\ZZ_k)\,.
}\end{equation}
This ends the proposition.
\end{proof}

In the case of stratifolds, the fundamental classes are defined by Tene \cite{tene2}. More precisely, 
let $S$ be a compact oriented regular $p$-stratifold of dimension $n$ and denote by $(M,\partial M)$ the smooth manifold we attach as top stratum. We have isomorphisms 
\begin{equation}
\xymatrix{H_n(M,\partial M)\ar[r]^\cong_{exc}& H_n(S,S_{n-2})&H_n(S)\ar[l]_(0.38){\cong}\,,}
\end{equation}
where $S_{n-2}$ is the $(n-2)$-skeleton of $S$.
The {\bf fundamental class} $[S]\in H_n(S)$ is defined as the image of $[M,\partial M]\in H_n(M,\partial M)$.

Let $(T,\partial T)$ be a compact oriented regular $p$-stratifold of dimension $n+1$ with boundary and denote by $(B,\partial B)$ the smooth manifold with boundary and collar attached as the top stratum. 
\begin{equation}
\xymatrix{H_{n+1}(B,\partial B)\ar[r]^(0.4){\cong}_(0.4){exc}& H_{n+1}(T,T_{n-1}\cup \partial T)&H_{n+1}(T,\partial T)\ar[l]_(0.38){\cong}\,,}
\end{equation}
where $T_{n-1}$ is the $(n-1)$-skeleton of $T$. The {\bf relative fundamental class} $[T,\partial T]\in H_{n+1}(T,\partial T)$ is defined as the image of $[B,\partial B]\in H_{n+1}(B,\partial B)$.
\begin{proposicion}(\cite[Lemma~3.9]{tene2})
Let $T$ be a compact oriented regular stratifold of dimension $(n+1)$, where the boundary is $\partial T$. 
Then, the image of $[T,\partial T]$ under the map $\partial: H_{n+1}(T,\partial T)\longrightarrow H_n(\partial T)$ is the class $[\partial T]$.

\end{proposicion}

Assume $\mathcal{S}=(S,\delta S,\theta_i)$ is a closed $n$-dimensional $\ZZ_k$-stratifold where both $S$ and $\delta S$ are compact oriented regular $p$-stratifolds. Similarly as in diagram \eqref{eqod1}, we can find {\bf the fundamental class} $[\mathcal{S}]_{\ZZ_k}\in H_{n}(\tilde{S};\ZZ_k)$ using the following commutative diagram:
\begin{equation}\label{eqo123}\xymatrix{
 \cdots \ar[r] & H_{n}(\partial S;\ZZ_k) \ar[r]\ar[d]^{q_*}    & H_{n}(S;\ZZ_k) \ar[d]^{q_*}\ar[r]^{i_*}    & H_{n}(S,\partial S;\ZZ_k) \ar[d]^{q_*}\ar[r]^{\partial}    & H_{n-1}(\partial S;\ZZ_k) \ar[d]^{q_*}\ar[r]  & \cdots \\
\cdots \ar[r] & H_{n}(\tilde{\partial S};\ZZ_k) \ar[r]    & H_{n}(\tilde{S};\ZZ_k) \ar[r]^{i_*}   & H_{n}(\tilde{S}, \tilde{\partial S};\ZZ_k) \ar[r]^{\partial} & H_{n-1}(\tilde{\partial S};\ZZ_k) \ar[r]        & \cdots
}\end{equation}
In the previous diagram, the rows are the long exact sequences associated with the pairs $(S,\partial S)$ and $(\tilde{S},\tilde{\partial S})$. The quotient map induces the vertical morphisms.  
Again, we have the isomorphism $q_*:H_n(S,\partial S;\ZZ_k)\longrightarrow H_n(\tilde{S},\tilde{\partial S};\ZZ_k)$ and $H_n(\tilde{\partial S};\ZZ_k)\cong H_n(\delta S;\ZZ_k)=0$
The same arguments as those for $\ZZ_k$-manifolds, show that there exists a unique fundamental class $[\mathcal{S}]_{\ZZ_k}\in H_n(\tilde{S};\ZZ_k)$ with the property  \begin{equation}i_*([\mathcal{S}]_{\ZZ_k})=q_*([S,\partial S])\,.\end{equation}

The local orientations at each point define the fundamental class of a manifold. This property also follows for stratifolds considering points inside the interior of the top stratum. Therefore, we use this fact to generalize Lemma \ref{gato} for stratifolds. More precisely, let $S$ be a compact oriented regular $p$-stratifold of dimension $n$, which is the gluing $S=S'\sqcup_{\partial S'=\partial S''}S''$, then in the next diagram, we have that the  fundamental classes are mapped to the generators associated with the point $x$,  
\begin{equation}
\xymatrix{
[S'',\partial S'']\in
H_n(S'',\partial S'')\ar[r]^(0.5){\cong} & H_n(S'',(S'')_{n-2}\cup \partial S'')\ar[r]  & H_n(S'',S''-\{x\})\ar[dd]^\cong \\
\hspace{1.5cm}H_n(S,S')\ar@<-4ex>[u]^\cong_{exc} &&\\
[S]\in H_n(S)\ar@<-4ex>[u]\ar[r]^\cong & H_n(S,S_{n-2})\ar[r] & H_n(S,S-\{x\}) \,.}
\end{equation}
Here $(S'')_{n-2}$ and $S_{n-2}$ are the $(n-2)$-skeletons of $S''$ and $S$.

Similarly, we show the existence of a {\bf relative fundamental class} of an $(n+1)$-dimensional $\ZZ_k$-stratifold with boundary 
$\mathcal{T}=(T,\partial T,\psi_i)$. The $\ZZ_k$-boundary is $\partial \mathcal{T}=(S,\delta S,\theta_i)$ and all stra\-tifolds are compact oriented regular $p$-stratifolds. We can find the fundamental class $[\mathcal{T},\partial\mathcal{T}]_{\ZZ_k}$ using the following commutative diagram:
\begin{equation}\label{eq12}\xymatrix{
 \cdots \ar[r] & H_{n+1}(\partial T,S;\ZZ_k) \ar[r]\ar[d]^{q_*}    & H_{n+1}(T,S;\ZZ_k) \ar[d]^{q_*}\ar[r]^{i_*}    & H_{n+1}(T,\partial T;\ZZ_k) \ar[d]^{q_*}\ar[r]^{\partial}    & H_{n}(\partial T,S;\ZZ_k) \ar[d]^{q_*}\ar[r]  & \cdots \\
\cdots \ar[r] & H_{n+1}(\tilde{\partial T},\tilde{S};\ZZ_k) \ar[r]    & H_{n+1}(\tilde{T},\tilde{S};\ZZ_k) \ar[r]^{i_*}   & H_{n+1}(\tilde{T}, \tilde{\partial T};\ZZ_k) \ar[r]^{\partial} & H_{n}(\tilde{\partial T},\tilde{S};\ZZ_k) \ar[r]        & \cdots
}\end{equation}
where the rows are the long exact sequences associated with the triples $(T,\partial T, S)$ and $(\tilde{T},\tilde{\partial T}, \tilde{S})$, respectively, and the vertical morphisms are induced by considering the quotient spaces. 
The same arguments shows the existence of the fundamental class $[\mathcal{T},\partial\mathcal{T}]_{\ZZ_k}\in H_{n+1}(\tilde{T},\tilde{S};\ZZ_k)$ with the property  \begin{equation}i_*([\mathcal{T},\partial\mathcal{T}]_{\ZZ_k})=q_*([T,\partial T])\,.\end{equation}
The same arguments as those for $\ZZ_k$-manifolds, show that the image of $[\mathcal{T,\partial\mathcal{T}}]_{\ZZ_k}$ under the map $\partial: H_{n+1}(\tilde{T},\tilde{S};\ZZ_k)\longrightarrow H_n(\tilde{S};\ZZ_k)$ is the class $[\partial \mathcal{T}]_{\ZZ_k}$.

As a consequence, the following result is straightforward.
\begin{proposicion}There is a well defined natural transformation $\Phi':SH_*(X;\ZZ_k)\longrightarrow H_*(X;\ZZ_k)$ which fits into the commutative diagram 
\begin{equation} 
    \xymatrix{ \Omega_*(X;\ZZ_k) \ar[r]^{\Phi}\ar[d] & H_*(X;\ZZ_k) \,.\\ SH_*(X;\ZZ_k) \ar[ru]_{\Phi'} &   }
\end{equation}
In addition, $\Phi'$ is an isomorphism for all $CW$ complexes.
\end{proposicion}

\section{A geometric description of the Atiyah--Hirzebruch spectral sequence for $\ZZ_k$-coefficients}
\label{secc1}
We assume all spaces are $CW$ complexes, and for a $CW$ complex $X$, we denote by $X^k$ its $k^{th}$ skeleton. For a  generalized homology theory $h$, a Postnikov tower is a sequence of homology theories $h^{(r)}$ and natural transformations
\begin{equation}
\xymatrix{
&h\ar[d]\ar[rrd]\ar[rrrd]\ar[rrrrd]&&&&\\
\cdots\ar[r]&h^{(r)}\ar[r]&\cdots\ar[r]&h^{(2)}\ar[r]&h^{(1)}\ar[r]&h^{(0)}\,,}
\end{equation}
such that we have the properties:
\begin{itemize}
    \item $h_n(*)\rightarrow h_n^{(r)}(*)$ is an isomorphism for $n\leq r$, and
    \item $h_n^{(r)}(*)$ is trivial for $n>r$.
\end{itemize}
These properties determine $h^{(r)}$ completely, see \cite[Chapter~II, 4.13-4.18]{rudyak}.

Every generalized homology theory $h$, has an associated Atiyah--Hirzebruch spectral sequence  $(E^r_{s,t},d^r_{s,t})$. For $r\geq 2$, Tene \cite{tene} constructs a natural isomorphism of spectral sequences 
$$\xymatrix{E_{s,t}^{r}=\frac{\operatorname{Im}\left(h_{s+t}(X^s,X^{s-r})\longrightarrow h_{s+t}(X^s,X^{s-1})\right)}{\operatorname{Im}\left(h_{s+t+1}(X^{s+r-1},X^s)\longrightarrow h_{s+t}(X^s,X^{s-1})\right)}\ar[rr]&& {\hat{E}}_{s,t}^r=\operatorname{Im}\left(h_{s+t}^{(t+r-2)}(X^s)\rightarrow h_{s+t}^{(t)}(X^{s+r-1})\right)}\,.$$
We now explain with diagram \eqref{eqrefg} the argument of Tene \cite[Sec.~4]{tene}, that gives the isomorphisms
$$E_{s,t}^r=\frac{\operatorname{Im}(f')}{\operatorname{Im}(f)}\cong \operatorname{Im}(f_1)\cong \operatorname{Im}(f_2)\cong \operatorname{Im}(f_3)={\hat{E}}_{s,t}^r\,.$$
\begin{equation}\label{eqrefg}
\xymatrix{ & & h_{s+t}^{(t+r-2)}(X^s)\ar[r]^{f_3}\ar@{->>}[d]&h_{s+t}^{(t)}(X^{s+r-1})\ar[dd]^{\cong}\\ & &h_{s+t}^{(t+r-2)}(X^s,X^{s-r})\ar[rd]& \\ & h_{s+t}(X^s,X^{s-r})\ar[r]\ar[d]^{f'}\ar[dr]^{f_1}\ar@{->>}[ru]\ar@/^{7mm}/[rr]^{f_2}& h_{s+t}(X^{s+r-1},X^{s-r})\ar[r]\ar[d]&h_{s+t}^{(t)}(X^{s+r-1},X^{s-r})\ar@{>->}[d]\\
h_{s+t+1}(X^{s+r-1},X^s)\ar[r]_f\ar[ru] &h_{s+t}(X^s,X^{s-1})\ar[r] &h_{s+t}(X^{s+r-1},X^{s-1})\ar[r]_{\cong} & h_{s+t}^{(t)}(X^{s+r-1},X^{s-1})\,.}
\end{equation}
The differential $\hat{d}_{s,t}^r:{\hat{E}}_{s,t}^r\rightarrow {\hat{E}}_{s-r,t+r-1}^{r}$ is the homomorphism induced by the following diagram: 
\begin{equation}
\xymatrix{&h_{s+t}^{(t+r-2)}(X^s)\ar[r]\ar[d]^{\Phi}&h_{s+t}^{(t)}(X^{s+r-1})\ar[d]^{\Phi}\\&h_{s+t-1}(X^{s-r+1})\ar[r]\ar[d]^{\Psi}&h_{s+t-1}(X^{s-1})\ar[d]^{\Psi}\\h_{s+t-1}^{(t+2r-3)}(X^{s-r})\ar[r]&h_{s+t-1}^{(t+2r-3)}(X^{s-r+1})\ar[r]&h_{s+t-1}^{(t+r-1)}(X^{s-1})\,,}
\end{equation}
where the natural transformation $\Phi$ is defined by the composition 
$$h_n^{(r)}(X)\rightarrow h_n^{(r)}(X,X^{n-r-1})\stackrel{\cong}{\rightarrow} h_n(X,X^{n-r-1})\rightarrow h_{n-1}(X^{n-r-1})\,,$$
and $\Psi$ is the natural transformation given by the composition of the natural transformations in the Postnikov tower.

For oriented bordism $\Omega_*$, Tene \cite{tene} has a geometric description of the Atiyah--Hirzebruch spectral sequence, coming from  a geometric description of  Postnikov tower $SH^{(r)}$. This description of the spectral sequence is similar in spirit to the Conner-Floyd spectral sequence appearing in equivariant bordism \cite{CF} and the spectral sequence for orbifold corbordism of \cite{Angel-cobordism}.  
The bordism theory $SH^{(r)}$ is defined using oriented p-stratifolds, with all strata of codimension $0<k<r+2$ empty. 
Thus a singular stratifold $S$ in $X$, of the form $f: S \rightarrow X$, gives an element of $SH^{(r)}_n(X)$ if $S$ is an $n$-dimensional stratifold with singular part of dimension at most $(n-r-2)$. We put a similar restriction to the stratifold bordisms, which are $(n+1)$-dimensional stratifolds with boundary, and the singular part is of dimension at most $(n-r-1)$. 

Therefore, we have natural transformations $\Omega_n\rightarrow SH_n^{(r)}$, such that $\Omega_n(*)\rightarrow SH_n^{(r)}(*)$ are isomorphisms for $n\leq r$, and $SH_n^{(r)}(*)$ is trivial for $n>r$. Among other properties, we obtain that $SH_n^{(r)}(X^k)$ is trivial for $k+r< n$. 

For $r\geq 2$, denote \ben {\hat{E}}^{r}_{s,t}=\operatorname{Im}( SH^{(t+r-2)}_{s+t}( X^s ) \longrightarrow 
 SH^{(t)}_{s+t}( X^{s+r-1} ))\,,\een and the differential ${\hat{d}}^r_{s,t}:{\hat{E}}^r_{s,t}\longrightarrow {\hat{E}}^r_{s-r,t+r-1}$ is the homomorphism induced by the following diagram:
 \ben
\xymatrix{&SH_{s+t}^{(t+r-2)}(X^s)\ar[r]\ar[d]_\Phi&SH_{s+t}^{(t)}(X^{s+r-1})\ar[d]_\Phi\\
&\Omega_{s+t-1}(X^{s-r+1})\ar[d]_\Psi\ar[r]%\ar@{-->}[ld]
&\Omega_{s+t-1}(X^{s-1})\ar[d]_\Psi\\
SH_{s+t-1}^{(t+2r-3)}(X^{s-r})\ar[r]&SH_{s+t-1}^{(t+2r-3)}(X^{s-r+1})\ar[r]&SH_{s+t-1}^{(t+r-1)}(X^{s-1})
\,,}
\een 
where $\Phi$ is a natural transformation defined by 
\ben
SH_n^{(r)}(X)\rightarrow SH_n^{(r)}(X,X^{n-r-1})\stackrel{\cong}{\rightarrow}\Omega_n(X,X^{n-r-1})\rightarrow\Omega_{n-1}(X^{n-r-1})\,.
\een
The isomorphism $SH_n^{(r)}(X,X^{n-r-1})\stackrel{\cong}{\rightarrow}\Omega_n(X,X^{n-r-1})$ is the restriction to the top stratum and the map $\Omega_n(X,X^{n-r-1})\rightarrow\Omega_{n-1}(X^{n-r-1})$ is the boundary homomorphism. 
The natural transformation $\Psi$ is the composition of the natural transformations in the Postnikov tower.  
Therefore, for a
stratifold $S$ of dimension $s+t$, with a map $f:S\rightarrow X^s$, the image of the differential $d^r_{s,t}$ is induced by
\ben \label{wef}[f:S\rightarrow X^s]\longmapsto [f|_{\op{sing}(S)}\circ g:\partial W\rightarrow X^{s-1}]\,,\een
where $W$ is the top stratum of $S$ and $g:\partial W\rightarrow\op{sing}(S)$ is the attaching map used to glue $W$ to the singular part $\op{sing}(S)$.

The $\ZZ_k$-bordism groups $\Omega_n(X ;\ZZ_k)$ form a generalized homology theory (this follows by Section \ref{sec8} or see \cite[Chapter~III]{AGA}). The authors define bordism theory for resolutions with abelian groups in that book. The standard resolution for $\ZZ_k$ and the theory of this section coincide with that given by the definition of $\ZZ_k$-manifolds. We construct a Postnikov tower $SH^{(r)}(\cdot;\ZZ_k)$ defined  with oriented $\ZZ_k$-stratifolds, with all strata of codimension $0<k<r+2$ empty. Thus a singular $\ZZ_k$-stratifold in $X$, of the form $f:(S,\delta S)\longrightarrow X$, represents an element of $SH^{(r)}_n(X)$ if it satisfies:
\begin{itemize}
    \item $S$ is an $n$-dimensional $\ZZ_k$-stratifold with singular part of dimension at most $(n-r-2)$; and
    \item $\delta S$ is an $(n-1)$-dimensional $\ZZ_k$-stratifold with singular part of dimension at most $(n-r-3)$.
\end{itemize}
Similarly, the stratifold bordisms $(T,\delta T)$ should satisfy:
\begin{itemize}
    \item $T$ is an $(n+1)$-dimensional $\ZZ_k$-stratifold with boundary, the singular part is of dimension at most $(n-r-1)$; and 
    \item $\delta T$ is an $n$-dimensional $\ZZ_k$-stratifold with boundary, and the singular part is of dimension at most $(n-r-2)$.
\end{itemize}
Notice that we obtain $SH^{(0)}(\cdot ;\ZZ_k) = SH(\cdot ;\ZZ_k)$.
In what follows, we use the important property that $\Omega_*(*)$ has no odd torsion and just $2$-torsion, see \cite{milnor}.

\bteo\label{teto}
For $k$ an odd number, the homology theories $SH^{(r)}_*(\cdot ;\ZZ_k)$ give the Postnikov tower of the generalized homology theory $\Omega_*(\cdot ;\ZZ_k)$.
\eteo

\bdem
We have natural transformations 
\begin{equation}
\xymatrix{
&\Omega_*(\cdot ;\ZZ_k)\ar[d]\ar[rrd]\ar[rrrd]\ar[rrrrd]&&&&\\
\cdots\ar[r]&SH_*^{(r)}(\cdot ;\ZZ_k)\ar[r]&\cdots\ar[r]&SH_*^{(2)}(\cdot ;\ZZ_k)\ar[r]&SH_*^{(1)}(\cdot ;\ZZ_k)\ar[r]&SH_*^{(0)}(\cdot ;\ZZ_k)\,.}
\end{equation}
The conditions of the Postnikov tower are proven as follows:
\begin{itemize}
    \item  assume $n\leq r$, hence $(n-r-2) \leq -2$ and $(n-r-1) \leq -1$. Thus the $\ZZ_k$-stratifolds are $\ZZ_k$-manifolds and the $\ZZ_k$-stratifolds bordism are $\ZZ_k$-manifolds with boundary. Therefore, the maps 
    $\Omega_n(*,\ZZ_k) \to SH^{(r)}_n(*,\ZZ_k)$ are isomorphisms for $n\leq r$.
    
    \item assume $n>r +1$, hence $n-r-1 \geq 1$ and $n-r-2\geq 0$. Thus for an $n$-dimensional $\ZZ_k$-stratifold $(S,\delta S)$ in $SH^{(r)}_n(*;\ZZ_k)$, we construct the cone as in Definition \ref{elcono}. As a consequence, $SH^{(r)}_n(*;\ZZ_k) = 0$ for $n>r+1$. 
    
    \item assume $n= r+1$, hence $n-r-2=-1$ and $n-r-3=-2$. Thus an $n$-dimensional $\ZZ_k$-stratifold in $SH^{(r)}_n(*;\ZZ_k)$ is a $\ZZ_k$-manifold $(M,\delta M)$. Because $n-r-1=0$ and $n-r-2=-1$,
    we allow $\ZZ_k$-stratifold bordisms with singular points of dimension at most $0$ and the Bockstein has to be an $n$-dimensional manifold with boundary. In $\Omega_{n-1}(*)$ we have $k[\delta M]=0$, but since {\bf $\Omega_*$ has no odd torsion}, then there exists an $n$-dimensional manifold with boundary $N$ where $\partial N=\delta M$. Consider the $\ZZ_k$-stratifold bordism $(C(kN\sqcup_{\partial M}M),N)$ where $C(kN\sqcup_{\partial M}M)$ is the closed cone. The $\ZZ_k$-boundary is precisely the $\ZZ_k$-manifold $(M,\delta M)$ which shows that $SH^{(r)}_n(*;\ZZ_k) = 0$ for $n=r+1$.
\end{itemize}
For $k=2$, this argument fails, and we cannot work around it using  the cone of $\delta M$, because we obtain singular points of dimension $\geq 1$.
\edem

The same arguments of Tene \cite{tene} give a geometric description of the Atiyah--Hirzebruch spectral sequence for $\mathbb{Z}_k$-bordism. For $r\geq 2$ and $X$ a $CW$ complex, define 
\ben {\hat{E}}^{r}_{s,t}=\operatorname{Im}( SH^{(t+r-2)}_{s+t}( X^s ;\ZZ_k) \longrightarrow 
 SH^{(t)}_{s+t}( X^{s+r-1};\ZZ_k ))\,,\een 
 and the differential ${\hat{d}}^r_{s,t}:{\hat{E}}^r_{s,t}\longrightarrow {\hat{E}}^r_{s-r,t+r-1}$ 
is the homomorphism induced by the following diagram:
 \ben
\xymatrix{&SH_{s+t}^{(t+r-2)}(X^s;\ZZ_k)\ar[r]\ar[d]_\Phi&SH_{s+t}^{(t)}(X^{s+r-1};\ZZ_k)\ar[d]_\Phi\\
&\Omega_{s+t-1}(X^{s-r+1};\ZZ_k)\ar[d]_\Psi\ar[r]%\ar@{-->}[ld]
&\Omega_{s+t-1}(X^{s-1};\ZZ_k)\ar[d]_\Psi\\
SH_{s+t-1}^{(t+2r-3)}(X^{s-r};\ZZ_k)\ar[r]&SH_{s+t-1}^{(t+2r-3)}(X^{s-r+1};\ZZ_k)\ar[r]&SH_{s+t-1}^{(t+r-1)}(X^{s-1};\ZZ_k)
\,,}
\een 
Therefore, for a singular $\ZZ_k$-stratifold $((S,\delta S), f:S\rightarrow X^s)$, we consider the top stratum which is a $\ZZ_k$-manifold with boundary $(W,\delta W)$. Denote the $\ZZ_k$-boundary by $(M,\delta M):=\partial (W,\delta W)$ and $g:M\rightarrow \op{sing}(S)$ the attaching map used to glue $W$ to the singular part which is of dimension at most $s-r$. The image of the differential $d^r_{s,t}$ is induced by
\begin{equation}\label{diferential}
    [(S,\delta S), f:S\rightarrow X^s]\longmapsto [(M,\delta M),f|_{\op{sing}(S)}\circ g:M\rightarrow X^{s-r}]\,.
\end{equation}
We have finally proved:
\bteo
For $k$ an odd number, the filtration of the Atiyah--Hirzebruch spectral sequence of $\ZZ_k$-bordism
\begin{equation}  E_{n,0}^{\infty} \subseteq \dots \subseteq E^{r+2}_{n,0} \subseteq \dots   \subseteq E_{n,0}^{2} \cong H_n(X;\integer_k)\,, \end{equation}
coincides with 
\begin{equation}
    E_{n,0}^{r}=\op{Im}\left( SH_n^{(r-2)}(X;\ZZ_k)\rightarrow SH_n^{(0)}(X;\ZZ_k)\cong H_n(X;\ZZ_k)\right)\,,
\end{equation}
i.e., the set of classes generated by singular $\ZZ_k$-stratifolds in $X$ with singular part of dimension at most $n-r-2$.
\eteo
Notice that the Atiyah--Hirzebruch spectral sequence is trivial for $k=2$, hence the last theorem does not apply.

\section{Geometric representatives of non-representable classes}
\label{secGR}
The present section is motivated by the calculations of Koshikawa \cite{koshi} and the geometric constructions of the counterexamples of Thom found by the authors in \cite{AST}.

The Steenrod problem \cite{Eil} states the following: if $z\in H_n(X)$ is an integral homology class, does there exist an oriented manifold $M$ and a map $f : M \rightarrow X$ such that $z$ is the image of the generator of $H_n(M)$?

Conner--Floyd \cite{CF} rephrased the Steenrod realization problem in terms of the Atiyah--Hirzebruch spectral sequence $(E_{s,t}^r,d^r_{s,t})$. More precisely, the homomorphism from oriented bordism to integral homology $\Omega_*(X)\rightarrow H_*(X)$ is an epimorphism if and only if the differentials 
$d^r_{s,t}:E_{s,t}^r\longrightarrow E_{s-r,t+r-1}^r$ are trivial for all $r\geq 2$.

Using the previous section, the Steenrod realization problem for $\ZZ_k$-coefficients  
has the following form.

\begin{teorema}
If $X$ is a $CW$-complex and $k$ and odd number, then for the Atiyah--Hirzebruch spectral sequence $(E_{s,t}^r,d^r_{s,t})$, the differentials $d_{s,t}^r:E_{s,t}\rightarrow E_{s-r,t+r-1}$ are trivial for all $r\geq 2$ if and only if the map $\mu:\Omega_n(X;\ZZ_k)\longrightarrow H_n(X;\ZZ_k)$ is an epimorphism for all $n\geq 0$. 
\end{teorema}

For the rest of this section, we assume that $k$ is an odd prime number $p$. Following Conner--Floyd \cite{CF}, we identify stratifolds with maps to $B\ZZ_p$ with stratifolds with free actions of $\ZZ_p$.

The Bockstein exact sequence of $B\ZZ_p$ implies the isomorphisms 
\begin{equation}
H_{2n-1}(B\ZZ_p)\stackrel{\op{mod}p}{\cong}H_{2n-1}(B\ZZ_p;\ZZ_p)\textrm{ and }H_{2n}(B\ZZ_p;\ZZ_p)\stackrel{\beta}{\cong}H_{2n-1}(B\ZZ_p)
\end{equation}
for $n>0$. Take generators $\alpha_i\in H_i(B\ZZ_p;\ZZ_p)$, such that $\beta(\alpha_i)=\alpha_{i-1}$ for $i$ even, and $\beta(\alpha_i)=0$ for $i$ odd. 
The generator $\alpha_{2i}$ is determined by the identity $\beta(\alpha_{2i})=\alpha_{2i-1}$. From Conner--Floyd \cite[p.~144]{CF}, we know that the following equation holds in bordism of $B\ZZ_p$: 
    \ben\label{cofo}p\alpha_{2i-1}+[M^4]\alpha_{2i-5}+[M^8]\alpha_{2i-9}+\cdots=0\,, \een
for $i\geq2$. The manifolds $M^{4k}$, $k=1,2,\cdots$ are constructed inductively in \cite{CF}. 
Therefore, there is a compact oriented manifold $V^{2i}$, with a free action of $\mathbb{Z}_p$, such that 
\ben
\partial V^{2i} = pS^{2i-1}\cup (M^4\times  S^{2i-5} )\cup (M^8\times S^{2i-9} )\cup \cdots 
\een
Denote by $C(M^{4k})$ the cone of $M^{4k}$ and take the gluing of $V^{2i}$ with $C(M^4)\times S^{2i-5} \cup C(M^8)\times S^{2i-9} \cup \cdots$. The boundary of this construction is $pS^{2i-1}$ and therefore the Bockstein is $\alpha_{2i-1}$. We illustrate this construction in Figure \ref{CV}.

\begin{figure}
    \centering
    \includegraphics[scale=1.4]{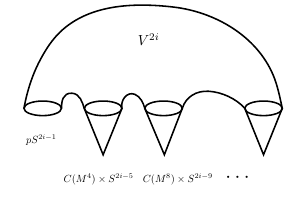}
    \caption{The class $\alpha_{2i}$}
    \label{CV}
\end{figure}

\begin{proposicion}
For $X=B\ZZ_p$, the differentials $d^r_{s,t}$ are trivial for $2\leq r\leq 4$. In particular, $d^5_{2i,0}:E^5_{2i,0}\rightarrow E^5_{2i-5,4}$ has the form 
\begin{equation}
d^5_{2i,0}:H_{2i}(X;\Omega_0(*;\mathbb{Z}_p))\longrightarrow H_{2i-5}(X;\Omega_4(*;\mathbb{Z}_p))\,.
\end{equation}
\end{proposicion}
\begin{proof}
As a consequence of the identities $\Omega_1(*;\mathbb{Z}_p)=0$, $\Omega_2(*;\mathbb{Z}_p)=0$ and $\Omega_3(*;\mathbb{Z}_p)=0$, we obtain that the Atiyah--Hirzebruch spectral sequence for $\mathbb{Z}_p$-bordism, satisfies $E^2\cong\cdots \cong E^5$, consequently, the first not trivial differential could be $d^5_{s,t}$.
\end{proof}

Previously, we constructed $\alpha_{2i}$ as the formal sum $V^{2i}+C(M^4)\alpha_{2i-5} +C(M^8)\alpha_{2i-9} + \cdots$ where the Bockstein is $\alpha_{2i-1}$. This is a $\ZZ_p$-stratifold where 
the singularities are $S^{2i-5} \cup S^{2i-9} \cup \cdots $ coming from the coning. The top stratum is the $\ZZ_p$-manifold \begin{equation}\left(V^{2i} \cup  (M^{4} \times [0,1])\times S^{2i-5}  \cup (M^{8} \times [0,1]) \times S^{2i-9} \cup \cdots, S^{2i-1}\right) \end{equation} 
with $\ZZ_p$-boundary given by $(M^{4} \times S^{2i-5}  \cup M^{8} \times S^{2i-9} \cup \cdots,\emptyset)$. The attaching map of the top stratum is given by the map 
\begin{equation}
V^{2i} \cup  (M^{4} \times [0,1])\times S^{2i-5}  \cup (M^{8} \times [0,1]) \times S^{2i-9} \cup \cdots    \longrightarrow  S^{2i-5} \times \{*\} \cup S^{2i-9} \times \{*\} \cup \cdots
\end{equation}
that collapses each $S^{2i-4k-1} \times M^{4k} \times \{0\}$ to $S^{2i-4k-1} \times \{*\}$ , the other parts of the cylinders are attached to $V$ (see figure \ref{CV}). 

\begin{teorema}
For $X=B\ZZ_p$, the image of the class $\alpha_{2i}\in H_{2i}(X;\ZZ_p)$, $i\geq 3$, under the differential $d^5:H_{2i}(X;\ZZ_p)\longrightarrow H_{2i-5}(X;\ZZ_p)$ is non trivial.
\end{teorema}
\begin{proof}
From Section \ref{secc1}, we see that to calculate $d^5(V^{2i}+C(M^4)\alpha_{2i-5} +C(M^8)\alpha_{2i-9} + \cdots)$ we need to understand the singularities and the attaching maps to calculate $d^5=\Psi \circ \Phi$ as a map
\ben
\xymatrix{
\operatorname{Im}\left( SH^{(3)}_{2i}\left( X^{2i};\ZZ_p \right) \rightarrow SH^{(0)}_{2i}\left( X^{2i+4} ;\ZZ_p\right)\right)
\ar[d]_{d^5}\\
\operatorname{Im}\left( SH^{(7)}_{2i-1}\left( X^{2i-5} ;\ZZ_k\right) \rightarrow 
 SH^{(4)}_{2i-1}\left( X^{2i-1} ;\ZZ_k\right)
\right)\,.}
\een

By definition, $\Phi$ is restricting the classifying map associated with the principal bundle of the action, to the boundary of the manifold attached as top stratum. Therefore
$$\Phi: \operatorname{Im}\left( SH^{(3)}_{2i}\left( X^{2i};\ZZ_p \right) \rightarrow 
 SH^{(0)}_{2i}\left( X^{2i+4} ;\ZZ_p\right)
\right) \longrightarrow \operatorname{Im}\left( \Omega_{2i-1}\left(X^{2i-4};\mathbb{Z}_p\right) \rightarrow \Omega_{2i-1}\left(X^{2i-1};\mathbb{Z}_p\right)  \right)$$
is the restriction of the classifying map associated to $M^{4}\times S^{2i-5}  \cup  M^{8}\times S^{2i-9} \cup \cdots $.
We write this as:
$$
\Phi(V^{2i}+C(M^4)\alpha_{2i-5} +C(M^8)\alpha_{2i-9} + \cdots) = M^{4}\times \alpha_{2i-5}  +   M^{8}\times\alpha_{2i-9} + \cdots 
$$
Now when we apply the natural transformation
$$\Psi: \operatorname{Im}\left( \Omega_{2i-1}\left(X^{2i-4};\mathbb{Z}_p\right) \rightarrow \Omega_{2i-1}\left(X^{2i-1};\mathbb{Z}_p\right)\right) \longrightarrow \operatorname{Im}\left( SH^{(7)}_{2i-1}\left( X^{2i-5};\ZZ_p \right) \rightarrow 
 SH^{(4)}_{2i-1}\left( X^{2i-1} ;\ZZ_p\right)
\right),
$$
in $SH^{(4)}_{2i-1}\left( X^{2i-1} ;\ZZ_p\right )$, we have 
$$
 M^{4}\alpha_{2i-5}  +   M^{8}\alpha_{2i-9} + \cdots  =   M^{4}\alpha_{2i-5}\,.$$
Because in $SH^{(4)}_{2i-1}\left( X^{2i-1} ;\ZZ_p\right )$ bordisms with singularities up to dimension $(2i-1)-4-1$ are allowed and for $k > 1$, $M^{4k}\alpha_{2i-4k-1} $ is the boundary of $C(M^{4k})\alpha_{2i-4k-1} $ which is a bordism with singular stratum of dimension $2i-4k-1 \leq (2i-1)-4-1$. Therefore 
$$
d^5(V+C(M^4)\alpha_{2i-5} +C(M^8)\alpha_{2i-9} + \cdots) =   M^4\alpha_{2i-5}\,,
$$
which is a generator of $H_{2i-5}(X;\Omega_4(*;\mathbb{Z}_p))$, since
$M^4$ can be taken to be $\mathbb{CP}^2$ which generates  $\Omega_4(*;\mathbb{Z}_p)$.
Consequently, the image of $\alpha_{2i}$ under the differential $d^5$ is a nontrivial element.
\end{proof}

\section{{$\ZZ_2$-}stratifold homology is stratifold homology with {$\ZZ_2$-}coefficients}
\label{secfin}
Kreck \cite[Ch.~4]{kreck} introduces the theory of $\mathbb{Z}_2$-oriented stratifolds in order to represent homology with $\ZZ_2$-coefficients.
He calls this theory {\bf stratifold homology with {$\ZZ_2$-}coefficients} which is denoted by $\mathcal{SH}_*(X;\mathbb{Z}_2)$. 
The elements are bordism classes of singular stratifolds where the stratum of codimension $1$ is empty, but there is no requirement of an orientation of the top stratum. 
There is a natural isomorphism %with homology with $\ZZ_2$-coefficients
\begin{equation}
    \mathcal{SH}_*(X;\mathbb{Z}_2)\longrightarrow H_*(X;\mathbb{Z}_2)\,,
\end{equation}
that for a singular stratifold $(S,f:S\rightarrow X)$, takes the pushforward of the fundamental class $[S]\in H_*(S;\mathbb{Z}_2)$. 

This article introduces the theory of $\ZZ_2$-stratifolds, which also represent homology with $\ZZ_2$-coefficients. This is called {\bf $\ZZ_2$-stratifold homology} which is denoted by $SH_*(X;\ZZ_2)$. 
The elements are $\ZZ_2$-bordism classes of singular $\ZZ_2$-stratifolds where the stratum of codimension $1$ is empty,
but we require an orientation of the top stratum. There is a natural isomorphism 
\begin{equation}
    SH_*(X;\mathbb{Z}_2)\longrightarrow H_*(X;\mathbb{Z}_2)\,,
\end{equation}
that for a singular $\ZZ_2$-stratifold $((S,\delta S),f:S\rightarrow X)$,  takes the pushforward of the fundamental class $[S]_{\ZZ_2}\in H_n(\tilde{S};\ZZ_2)$.

Therefore, we have the following commutative diagram
\begin{equation}
\xymatrix{
SH_*(X;\mathbb{Z}_2) \ar[rr]^q \ar[dr]_{\cong} & & \mathcal{SH}_*(X;\mathbb{Z}_2) \,.\ar[dl]^{\cong} \\
& H_*(X;\mathbb{Z}_2) & }
\end{equation}
To define the map $q$, note that for an $n$-dimensional $\ZZ_2$-stratifold $(S,\delta S,\theta_i)$, the quotient space $\tilde{S}$ is an $n$-dimensional $\ZZ_2$-oriented stratifold. This is true because the two disjoint collars associated with the two embedded copies of the Bockstein $\delta S$ are combined to produce a bicollar on the quotient space $\tilde{S}$. For $(\mathcal{S},f)$ an $n$-dimensional singular $\ZZ_2$-stratifold, with $\mathcal{S}=(S,\delta S,\theta_i)$, we have the map $q:SH_n(X;\ZZ_2)\longrightarrow\mathcal{SH}_n(X;\ZZ_2)$ defined by $q([\mathcal{S},f])=[\tilde{S},\tilde{f}]$ where $\tilde{f}$ is the quotient map.

The description of the inverse for the isomorphism $q:SH_*(X;\mathbb{Z}_2) \longrightarrow\mathcal{SH}_*(X;\mathbb{Z}_2)$,
is an open question. Wall \cite{wall} shows a description for an $n$-dimensional manifold, whose first Stiefel-Whitney class $\omega_1\in H^1(M;\ZZ_2)$ is the restriction $\op{mod} 2$ of a class with integer coefficients. Thus there is a map $f:M\rightarrow K(\ZZ,1)=S^1$, which can be approximated by a smooth map. Take a regular value $t$ and consider the cutting $f^{-1}(t)$. The manifold with boundary $M-f^{-1}(t)$ is orientable, and in that case $f^{-1}(t)$ is also orientable, this describes $q^{-1}$ for this particular case.

\section{Appendix}

\label{apendice}

\subsection{Regular values for $\ZZ_k$-stratifolds}
\label{sec5}
In \cite[p.~27]{kreck}, Kreck defines a {\bf regular value} for a smooth map $f:S\rightarrow N$ from a closed stratifold $S$ to a boundaryless manifold $N$ as a point $x\in N$ such that for all $y\in f^{-1}(x)$ the differential $df_y$ is surjective, or, equivalently,  $x$ is a regular value of $f|_{S^i}$ for all $i$. Kreck \cite[Prop.~2.6 and Prop.~2.7, p.~27-29]{kreck} shows that the set of regular values of $f$ is dense in $N$, and $f^{-1}(x)$ is a stratifold of dimension $\op{dim}S-\op{dim} N$.

In \cite[p.~35]{kreck}, Kreck defines a {\bf smooth map} from a stratifold with boundary $T$ to a boundaryless manifold $N$, $f:T\rightarrow N$, as a continuous function whose restriction to $\stackrel{\circ}{T}=T-\partial T$ and to $\partial T$ is smooth and which commutes with the collar
$c:\partial T\times [0,\epsilon )\rightarrow U$, i.e., there is a $\delta >0$ with $\delta\leq\epsilon$ such that $fc(x,t)=f(x)$ for all $x\in \partial T$ and $t<\delta$.
Kreck \cite[p.~38]{kreck} says $x\in N$ is a {\bf regular value} if $x$ is a regular value for $f|_{T-\partial T}$ and $f|_{\partial T}$.
In this case, the preimage $f^{-1}(x)$ is a stratifold with boundary of dimension $\op{dim}T-\op{dim} N$. 
This fact is a generalization of a result of \cite[Prop.~2.7]{kreck} using local retractions for $T-\partial T$ and $\partial T$, together with Theorem \ref{polo}. Also, by Theorem \ref{polo}, the set of regular values is dense in $N$.

\begin{teorema}\cite[p.~60-62]{pollack}\label{polo}
Let $f:M\rightarrow N$ be a smooth map of a manifold $M$ with boundary onto a boundaryless manifold $N$ and let $x\in N$ a regular value of both $f$ and $\partial f$. Then the preimage $f^{-1}(x)$ is a submanifold of $M$ with boundary $f^{-1}(x)\cap \partial M$ of dimension $\op{dim}M-\op{dim}N$. Moreover, the set of critical values of both $f$ and $\partial f$ has measure zero.
\end{teorema}

In what follows, we obtain the version for stratifolds with boundary, of Kreck's Proposition 4.2 and Proposition 4.3 in \cite{kreck}.

\begin{proposicion}\label{regular}
Let $T$ be an oriented, regular stratifold with boundary, $f:T\rightarrow \RR$ a smooth function, and $t$ a regular value. Then $f^{-1}(t)$ is an oriented, regular stratifold with boundary. 
\end{proposicion}
\begin{proof}
We use the work of Kreck \cite[Prop.~4.2, p.~44]{kreck} in order to show that ${f|_{T-\partial T}}^{-1}(t)$ and ${f|_{\partial T}}^{-1}(t)$ are regular stratifolds. 
We induce the collar by restriction. We notice $f^{-1}(t)$ is  an oriented stratifold, since $T^{n-1}=\emptyset$ and the intersection with the top stratum is an oriented manifold.
\end{proof}

\begin{observacion}
In the case $T$ is a $p$-stratifold with boundary, see Remark \ref{pstratifold}; hence the preimage $f^{-1}(t)$ is also a $p$-stratifold with boundary, for $t$ a regular value. 
The construction of this $p$-stratifold is as follows: for $t$ a regular value, on each stratum $T_i$, , the preimage $f|_{T_i}^{-1}(t)$ is a submanifold of $T_i$ with boundary $f|_{T_i}^{-1}\cap\partial T_i$ by Theorem \ref{polo}. Similarly, the preimage $\partial f|_{\partial T_i}^{-1}(t)$ is a submanifold of $\partial T_i$. Moreover, these submanifolds come with collars and  attaching maps, that construct this $p$-stratifold with boundary inductively.
\end{observacion}

\begin{proposicion}\label{piropo}
Let $T$ be a regular stratifold with boundary. Then the set of regular points of a smooth map $f:T\rightarrow \RR$ 
is an open subset of $T$. If, in addition, $T$ is compact, the regular values form an open set.
\end{proposicion}
\begin{proof}
We know the regular points of $f|_{T - \partial T}$ and $f|_{\partial T}$ are open in $T-\partial T$ and $\partial T$, respectively. By definition $fc(x,t)=f(x)$ for some collar $c$ in $T$. So the regular points of $f|_{\partial T}$ extend to the collar by an open set. Thus we obtain the first statement.     
Now, in the case $T$ is compact, the singular points that are the complement of the regular points, form a closed set which is compact. Thus the image under $f$ is closed, implying that the regular values are form an open set.
\end{proof}

A crucial fact for the Mayer--Vietoris sequence for stratifolds is the following:

\begin{proposicion}(\cite[Prop.~2.8]{kreck})\label{pipi1} Let $S$ be a closed $n$-dimensional, connected stratifold and $A$ and $B$ disjoint closed non-empty subsets of $S$. Then there is a non-empty $(n-1)$-dimensional stratifold $P$ with $P\subset S-(A\cup B)$. That is, $P$ separates $A$ and $B$.
\end{proposicion}

\begin{observacion}\label{klklk}
More precisely, Kreck \cite[Prop.~2.4, p.~26]{kreck} constructs a smooth function $f:S\rightarrow\RR$ which maps $A$ to $1$ and $B$ to $-1$. 
The stratifold $P$ is the preimage $f^{-1}(t)$ of a regular value $t\in(-1,1)$ such that $f^{-1}(t)\subset S-(A\cup B)$ and $A\subset f^{-1}(t,\infty)$ and $B\subset f^{-1}(-\infty, t)$. After composition with an appropriate translation, we can assume $t=0$.
\end{observacion}

We extend Proposition \ref{pipi1} to the theory of $\ZZ_k$-stratifolds. However, it is not enough to consider stratifolds with boundary. The reason is that the smooth function must be $\ZZ_k$-invariant on the boundary. 
One needs a smooth function that factors as follows: 
$$\xymatrix{S \ar[rr]^f\ar[rd]_{pr}&&\RR\,.\\&\tilde{S}\ar[ru]_{\tilde{f}}&}$$

We need a  $\ZZ_k$-stratifold version of the following result.
\begin{proposicion}(\cite[Prop.~2.4]{kreck})\label{shishi}
Let $A\subset S$ be a closed subset of a stratifold $S$, $U$ an open neighborhood of $A$, and $f:U\rightarrow \RR$ a smooth function. Then there is a smooth function $g:S\rightarrow \RR$ such that $g|_A=f|_A$.
\end{proposicion}

\begin{proposicion}
Let $\mathcal{S} = (S,\delta S ,\theta_i)$ be an $n$-dimensional compact closed $\ZZ_k$-stratifold, $A \subseteq \tilde{S}$ a closed subset of the quotient space, $U$ an open neighborhood of $A$ and $f: U \to \mathbb{R}$ a smooth function. Then there exists a smooth function $G:S \to \RR$ that factors through the quotient space $\tilde{S}$ such that $G|_{A} = f|_{A}$ in the quotient space. 
\end{proposicion}

\begin{proof}
We construct a smooth function on $S$, which is the gluing of the following two functions:
\begin{itemize}
\item For the first function, consider $\delta S$ inside the quotient space $\tilde{S}$. By normality of $S$, there exists a closed subset $A_1\subset\delta S$ such that $A\cap \delta S\subset \operatorname{int} A_1$ and $A_1\subset \delta S\cap U$. 
By compactness and using the collar, $pr:\delta S\times [0,\epsilon)\rightarrow \tilde{S}$, we find $0<t<\epsilon$ such that 
$$pr^{-1}(A)\cap (\partial S\times [0,2t))  \subset pr^{-1}(A_1) \times [0,2t)\subset pr^{-1}(U)\,.$$
Proposition \ref{shishi} implies that it is possible to construct a smooth function $f_1:\delta S\rightarrow\RR$ that maps $A_1$ to $1$
and $f_1(x)=0$ for $x\in \delta S-U\cap\delta S$. Lift $f_1$ to a smooth function on the whole boundary $\partial S$ and take the smooth function $g_1:\partial S\times [0,2t)\rightarrow\RR$ by denoting $g_1(x,s)=f_1(x)$; 
\item For the second function takes the stratifold $S_1:=S-(\partial S\times [0,t])$ and again by Proposition \ref{shishi}, we can construct a smooth function $g_2: S_1\rightarrow\RR$ which maps $A\cap S_1$ to $1$ and $g_2(x)=0$ for $x\in S_1-U\cap S_1$.
\end{itemize}
A partition of unity glues these two functions together into a smooth function $G:S\longrightarrow \RR$, which is $\ZZ_k$-invariant. 
Thus it descends to the quotient and sends $A$ to $1$ and $\tilde{S}-U$ to 0. 
Using Kreck's Proposition 2.4 (Proposition \ref{shishi}), we apply the previous process to construct the function
 $G:S\longrightarrow \RR$, which is $\ZZ_k$-invariant and $G|_{A} = f|_{A}$ in the quotient space. 
\end{proof}

In conclusion, we obtain the $\ZZ_k$-stratifold version of Kreck's Proposition 2.8 in \cite{kreck} (Proposition \ref{pipi1}).

\begin{proposicion}\label{pipi2} Let $(S,\delta S)$ be a $n$-dimensional, compact, connected $\ZZ_k$-stratifold and $A$ and $B$ disjoint closed non-empty subsets of the quotient space $\tilde{S}$. Then there is a non-empty $(n-1)$-dimensional $\ZZ_k$-stratifold $(P,\delta P)$ with $\tilde{P}\subset \tilde{S}-(A\cup B)$ and $\delta P\subset \delta S-\left((A\cup B)\cap\delta S\right)$. 
\end{proposicion}

We construct a smooth function $G:S\rightarrow \RR$, that factors through the quotient space $\tilde{S}$, and maps $A$ to $1$ and $B$ to $-1$. The $\ZZ_k$-stratifold $(P,\delta P)$ is provided by a regular value $t\in(-1,1)$ of both $S$ and $\partial S$, and we have $P=G^{-1}(t)$ and $\delta P=G|_{\delta S}^{-1}(t)$.
The pair $(P,\delta P)$ is a $\ZZ_k$-stratifold because we choose a regular value by Proposition \ref{piropo} and the preimage
$P=G^{-1}(t)$ is a stratifold with boundary, where $\partial P=G^{-1}(t)\cap \partial S=\bigsqcup_{i\in\ZZ_k} \theta_i(G^{-1}(t)\cap \delta S)=\bigsqcup_{i\in\ZZ_k} \theta_i(G|_{\delta S}^{-1}(t))$ and the Bockstein is $\delta P=G|_{\delta S}^{-1}(t)$.

\subsection{Mayer--Vietoris sequence}\label{MV}
Let $U$ and $V$ be open subsets of a space $X$.
The purpose of this section is to show that the following long exact sequence is exact
\begin{equation}\label{prinseq}
\xymatrix{ \cdots \ar[r]^(.3)d & SH_n(U\cap V;\ZZ_k) \ar[r]^(.4){i_*} & SH_n(U;\ZZ_k) \oplus  SH_n(V;\ZZ_k)\ar[r]^(.6){j_*} & SH_n(U\cup V;\ZZ_k)
                \ar@{->} `r/8pt[d] `/10pt[l] `^dl[lll]|{\,d\,} `^r/12pt[dll] [dll] \\
             & SH_{n-1}(U\cap V;\ZZ_k)\ar[r]^(.4){i_*} & \cdots\hspace{3cm} & 
               }
\end{equation}
for $n\geq 1$. Denote by $i_{U}$ and $i_{V}$ the inclusions $U\cap V\hookrightarrow U$ and $U\cap V\hookrightarrow V$. 
Denote by $j_{U}$ and $j_{V}$ the inclusions $U\hookrightarrow U\cup V$ and $V\hookrightarrow U\cup V$. 
We describe the homomorphisms as follows:
\begin{itemize}
    \item $i_*:SH_n(U\cap V;\integer_k)\To SH_n(U;\integer_k)\oplus SH_n(V;\integer_k)$ is given by $({i_U}_*, {i_V}_*)$,
    \item $j_*:SH_n(U;\integer_k) \oplus SH_n(V;\integer_k) \To SH_n(U \cup V;\integer_k)$ is given by ${j_U}_* - {j_V}_*$; and 
    \item the connecting homomorphism  $d:SH_n(U\cup V;\ZZ_k)\longrightarrow SH_{n-1}(U\cap V;\ZZ_k)$, considers an element $[(S,\delta S),g]\in SH_n(U\cup V;\ZZ_k)$. For the projection $pr:S\rightarrow\tilde{S}$, we obtain 
    disjoint closed subsets of $\tilde{S}$ given by $A:=pr(g^{-1}(X-V))$ and $B:=pr(g^{-1}(X-U))$. By Proposition \ref{pipi2}, we obtain an $(n-1)$-dimensional $\ZZ_k$-stratifold $(P,\delta P)$ such that 
$\tilde{P}\subset \tilde{S}-(A\cup B)$ and $\delta P\subset \delta S-\left((A\cup B)\cap\delta S\right)$. We define 
\begin{equation}
    d([(S,\delta S),g])=[(P,\delta P),g|_{P}]\,.
\end{equation}
In the case $A$ or $B$ is empty, the $\ZZ_k$-stratifold $(P,\delta P)$ is empty, and the differential is zero.
\end{itemize}

\begin{proof}[Proof that $d$ is well-defined]
It was pointed out by Kreck \cite[p.~304]{kreck1} that in the case of bordism of smooth manifolds, the connecting homomorphism for the Mayer--Vietoris sequence is well-defined because  of the existence of a {\bf bicollar} for $P:=G^{-1}(0)$, i.e., an isomorphism with $P\times (-\epsilon,\epsilon)$, where $0$ is a regular value by the composition of a translation. For a stratifold $S$, this is only possible up to bordism where we naively change $S$ by $S-P\cup (P\times (-\epsilon,\epsilon))$. The formal statement is Lemma B.1 in \cite[p. 197]{kreck}, and the proof is as follows. Kreck's Proposition 4.3 in \cite{kreck} (Proposition \ref{piropo} for our case) allows us to choose $\delta>0$ such that $(-\delta,\delta)$ consists only of regular values of $G$. Consider a monotone smooth map $\mu: \RR \to \RR$ which is the identity for $|t| > \delta /2$ and $0$ for $|t| < \delta/4$. Take $\eta:S\times \RR\longrightarrow \RR$ mapping $(x,t)\mapsto G(x)-\mu(t)$ which has $0$ as regular value.
Kreck's Proposition 4.2 in \cite{kreck}, implies that $S'=\eta^{-1}(0)$ is a regular stratifold containing $P\times (-\delta/4,\delta/4)$ which is the required bicollar. It remains to construct a bordism between $S$ and $S'$.
Now, take the function $\gamma:S\times \RR\times [0,1]\longrightarrow \RR$ defined by $$(x,t,s)\mapsto G(x)-(\zeta(s)\mu(t)+(1-\zeta(s))t)\,,$$
where $\zeta:[0,1]\rightarrow \RR$ is 0 near 0 and 1 near 1. This map has 0 as a regular value, and the preimage $Q:=\gamma^{-1}(0)$ is the bordism between $S$ and $S'$.

For the case of $\ZZ_k$-stratifolds, we start with a closed $\ZZ_k$-stratifold $(S,\delta S,\theta_i)$ and we need to separate this $\ZZ_k$-stratifold by a bicollar over the regular $\ZZ_k$-stratifold 
$$(P,\delta P,\theta_i|_{\delta P})=(G^{-1}(0),G|_{\delta S}^{-1}(0),\theta_i|_{G|_{\delta S}^{-1}(0)})\,.$$ 
In such a case,  
the bicollar consists of a pair of embedded cylinders $(G^{-1}(0)\times (-\epsilon,\epsilon),G|_{\delta S}^{-1}(0)\times (-\epsilon,\epsilon))$ such that they are consistent with respect to the embeddings $\theta_i$'s. In order to reproduce Kreck's Lemma B.1 in \cite{kreck}, in the context of $\ZZ_k$-stratifolds, we observe that the map $\eta:S\times \RR\rightarrow\RR$
is $\ZZ_k$-invariant for our case and we take $(S',\delta S'):=\left(\eta^{-1}(0),\eta|_{\delta S\times \RR}^{-1}(0)\right)$
which is a regular $\ZZ_k$-stratifold by Proposition \ref{regular}. 
We construct the bicollar taking $\epsilon=\delta /4$ with $\delta$ as in the previous paragraph. The $\ZZ_k$-bordism between $((S,\delta S),\operatorname{id})$ and $((S',\delta S'),\pi_1)$, where $\pi_1$ is the projection on the first variable, is constructed similarly as in the case of stratifolds. 

The remaining steps to show that $d$ is well defined are analogous to the case of stratifolds \cite[p.~199-200]{kreck}.
The idea is to assume that $[(S,\delta S),g]$ is trivial, then $[(S',\delta S'),g\circ \pi_1]$ is also trivial.
For the modified $\ZZ_k$-stratifold $(S',\delta S')$, we can take the separating function given by the projection on the second variable. This means that there exists a $\ZZ_k$-bordism $(T,\delta T)$ that has as $\ZZ_k$-boundary $(S',\delta S')$. Moreover,  the separating function extends to $T$. This function has a regular value $t$ very close to $0$, then $(P\times \{t\},\delta P\times \{t\})$ is null $\ZZ_k$-bordant taking the preimage of $t$. However, this last $\ZZ_k$-stratifold is $\ZZ_k$-bordant to $(P,\delta P)$. This finishes the proof.
\end{proof}

The following results are required to show that \eqref{prinseq} is exact.

\bprop\label{l1}
Suppose $M$ is a manifold with boundary of dimension $n$ and $g:M\rightarrow \RR$ a smooth map with regular value $0$. 
Then the preimage $g^{-1}(-\infty,0]$ is a manifold with boundary, and the boundary has the form 
 $$g^{-1}(0) \sqcup_{(g^{-1}(0)\cap \partial M)} (g^{-1}(-\infty,0]\cap \partial M)\,.$$ 
In addition, if $M$ is oriented, then $g^{-1}(-\infty, 0]$ is oriented.
\eprop
\begin{proof}
Here we will dismiss the orientation of the manifolds, which is understood depending on the case. 
From \cite[p.~62]{pollack}, we have that for a manifold $N$ without boundary and $f:N\rightarrow\RR$ a smooth map, the preimage $f^{-1}(-\infty,0]$ is a manifold with boundary given by $f^{-1}(0)$. 
Thus the restriction to the boundary $g|_{\partial M}$ satisfies that $g|_{\partial M}^{-1}(-\infty,0]=g^{-1}(-\infty,0]\cap\partial M$ is a manifold with boundary $g|_{\partial M}^{-1}(0)=g^{-1}(0)\cap\partial M$. 
Furthermore, we use Theorem \ref{polo} (or \cite[p.~60-62]{pollack}) which shows that $g^{-1}(0)$ is also a manifold with boundary $g^{-1}(0)\cap\partial M$. 
Then we glue these two manifolds obtaining a boundaryless smooth manifold of dimension $n-1$. In Figure \ref{figuero}, we give an illustration of the boundary of $g^{-1}(-\infty,0]$.
\begin{figure}
    \centering
    \includegraphics[scale=1.6]{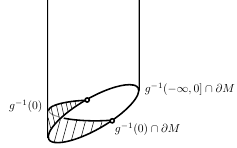}
    \caption{The boundary of $g^{-1}(-\infty,0]$.}
    \label{figuero}
\end{figure}
Now, we consider the restriction $g|_{M-\partial M}$ and we obtain a smooth structure for $g|_{M-\partial M}^{-1}(-\infty,0]=g^{-1}(-\infty,0]-(g^{-1}(-\infty,0]\cap \partial M)$ with boundary $g^{-1}(0)-(g^{-1}(0)\cap\partial M)$.
We can establish a collar around $g^{-1}(0)$. Since $g$ commutes with the collar of $\partial M$, there is a collar around $(g^{-1}(-\infty,0]\cap \partial M)$. 
Finally, similar as in the proof of Proposition \ref{equivalence}, we combine the two collars of $g^{-1}(0)$ and $(g^{-1}(-\infty,0]\cap \partial M)$ where we smooth the corners by straightening the angle \cite[p.~9-10]{CF} (or see Section \ref{sec4}).
Thus the proposition follows.
\end{proof}

Proposition \ref{l1} follows for stratifolds with boundary (all $p$-stratifolds). Notice $g^{-1}(-\infty,0]\cap \partial T=(g^{-1}(-\infty,0]\cap S)\cup (g^{-1}(-\infty,0]\cap k\delta T)$ and hence $g^{-1}(-\infty,0]$ is a stratifold with boundary where 
$$\partial g^{-1}(-\infty,0]=g^{-1}(0)\cup (g^{-1}(-\infty,0]\cap S)\cup (g^{-1}(-\infty,0]\cap k\delta T)\,.$$
Thus we obtain the following application for $\ZZ_k$-stratifolds.

\bcor\label{regular3}
Suppose $(T,\delta T)$ is a $\ZZ_k$-stratifold with boundary of dimension $n$ where the $\ZZ_k$-boundary is denoted by $(S,\delta S)$. Let $g:T\to \mathbb{R}$ be a smooth map which factors to the quotient space $\tilde{T}$
with $0$ as a regular value for $g$.
Then, the preimage $\left(g^{-1}(-\infty,0],g^{-1}(-\infty,0]\cap \delta T\right)$ is a $\ZZ_k$-stratifold with $\ZZ_k$-boundary, the following  $\ZZ_k$-stratifold 
$$\left(
g^{-1}(0)\cup (g^{-1}(-\infty,0]\cap S),
(g^{-1}(0)\cap\delta T)\cup
(g^{-1}(-\infty,0])\cap \delta S)\right)\,.$$
\ecor

Now, we use these tools to show the exactness of the Mayer--Vietoris sequence.

\begin{proof}[Proof of the exactness of \eqref{prinseq}]
We will follow the arguments used for the case of stratifolds \cite[p.~200-208]{kreck}, where we will specify the additional details used for the case of $\ZZ_k$-stratifolds.

To show that we have a complex, we notice that both ${j_U}\circ i_{U}$ and ${j_V}\circ i_{V}$ are the canonical inclusion $U\cap V\hookrightarrow U\cup V$, therefore $j_*\circ i_*=0$. 
We show the other cases $i_*\circ d=0$ and $d\circ j_*=0$ in what follows:
for the first identity, we choose a representative for the homology class (with $\ZZ_k$-coefficients) in $U\cap V$ such that we can cut along the separating $\ZZ_k$-stratifold defining the boundary operator. The two pieces separated by this $\ZZ_k$-stratifold induce the null $\ZZ_k$-bordisms on the homology groups (with $\ZZ_k$-coefficients) associated with $U$ and $V$.
For the second identity, if $[(S,\delta S),g]\in SH(U;\ZZ_k)$, we can choose a smooth function and the regular value such that the separating regular $\ZZ_k$-stratifold is empty, therefore, $d({j_U}_*)=0$. By the same argument $d({j_V}_*)=0$.

Now, we show  exactness:
\begin{itemize}
    \item $\op{ker}j_*\subset \op{im} i_*$. Consider $[(S,\delta S),f]\in SH_n(U;\ZZ_k)$ and $[(S',\delta S'),f']\in SH_n(V;\ZZ_k)$ with ${j_U}_*([(S,\delta S),f])={j_V}_*([(S',\delta S'), f'])$. There exists a $\ZZ_k$-bordism $((T,\delta T),F)$ between $[(S,\delta S),j_Uf]$ and $[(S',\delta S'),j_Vf']$ where $F=\tilde{F}\circ pr$ for the quotient $\tilde{F}:\tilde{T}\rightarrow U\cup V$. For the closed disjoint subsets $A_T=\tilde{S}\cup \tilde{F}^{-1}(X-V)$ and $B_T=\tilde{S'}\cup \tilde{F}^{-1}(X-U)$, we construct a separating function $G:T\rightarrow \RR$ which is $\ZZ_k$-invariant with $G(A_T)=1$ and $G(B_T)=-1$ and $-1<s<1$
 a regular value (we can assume $s=0$) such that $(G^{-1}(0),G^{-1}(0)\cap \delta T)$ is a separating $\ZZ_k$-stratifold. We can find a bicollar around $G^{-1}(0)$ similarly as when we show that $d$ is well defined. Therefore, Corollary \ref{regular3} implies that $((S,\delta S),f)$ and $((G^{-1}(0),G^{-1}(0)\cap \delta T),F|_{G^{-1}(0)})$ are $\ZZ_k$-bordant in $U$ by the $\ZZ_k$-bordism $((G^{-1}[0,\infty),G^{-1}[0,\infty)\cap\delta T),F|_{G^{-1}[0,\infty)})$, and $((G^{-1}(0),G^{-1}(0)\cap \delta T),F|_{G^{-1}(0)})$ and $((S',\delta S'),f')$ are $\ZZ_k$-bordant in $V$ by the $\ZZ_k$-bordism
$((G^{-1}(-\infty,0],G^{-1}(-\infty,0]\cap\delta T),F|_{G^{-1}(-\infty,0]})$. Thus ${i_U}_*([(G^{-1}(0),G^{-1}(0)\cap \delta T),F|_{G^{-1}(0)}])=[(S,\delta S),f]$ and 
${i_V}_*([(G^{-1}(0),G^{-1}(0)\cap \delta T),F|_{G^{-1}(0)}])=[(S',\delta S'),f']$.

    \item $\op{ker}i_*\subset \op{im} d$. Consider $[(P,\delta P),r]\in SH_{n-1(U\cap V;\ZZ_k)}$ with the identities ${i_U}_*([(P,\delta P),r])=0$ and ${i_V}_*([(P,\delta P),r])=0$. There exist null $\ZZ_k$-bordisms $((T_1,\delta T_1),R_1)$ and $((T_2,\delta T_2),R_2)$ of ${i_U}_*([(P,\delta P),r])$ and ${i_V}_*([(P,\delta P),r])$, respectively. We construct $((T_1\sqcup_P T_2,\delta T_1\sqcup_{\delta P} \delta T_2), R_1\sqcup_r R_2)$ with image under $d$ equal to $[(P,\delta P),r]$.

    \item $\op{ker}d\subset \op{im}j_*$. Consider $[(S,\delta S),f]\in SH_n(U\cup V;\ZZ_k)$ with $d([(S,\delta S),f])=0$. For a separating function $G$ with regular value $s$ as in the definition of $d$. We denote $(P,\delta P)=(G^{-1}(s),G|_{\delta S}^{-1}(s))$ which has a bicollar. We put $(S_+,\delta S_+)=(G^{-1}[s,\infty),G|_{\delta S}^{-1}[s,\infty))$ and $(S_-,\delta S_-)=(G^{-1}(-\infty,0],G|_{\delta S}^{-1}(-\infty,0])$. Then $S=S_+\sqcup_P S_-$ and $\delta S=\delta S_+\sqcup_{\delta P}\delta S_-$. By the assumptions, there is $((Z,\delta Z),r)$ with $r:Z\rightarrow U\cap V$ which has the $\ZZ_k$-boundary 
    $(P,\delta P)$ and $f|_P=r|_P$. Consider the continuous maps $f_+:S_+\sqcup_P Z\rightarrow U$, $f_-:S_-\sqcup_P Z\rightarrow V$ 
    and the gluing $T:=((S_+\sqcup_P Z)\times [0,1])\sqcup_Z((S_-\sqcup_P Z)\times [1,2])$ (similarly for the Bockstein $\delta T$) gives a $\ZZ_k$-bordism between 
    ${j_U}_*(((S_+\sqcup_P Z,\delta S_+\sqcup_{\delta P} \delta Z),f_+))-{j_V}_*(((S_-\sqcup_P Z,\delta S_-\sqcup_{\delta P} \delta Z),f_-)$ and $((S,\delta S),f)$.
    We show an illustrative picture of the $\ZZ_k$-bordism $(T,\partial T)$ in Figure \ref{dafor}.
    \begin{figure}[h!]
        \centering
        \includegraphics[scale=1.3]{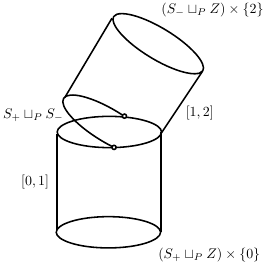}
        \caption{The $\ZZ_k$-bordism $T$.}
        \label{dafor}
    \end{figure}
\end{itemize}

\end{proof}

\end{document}